\documentclass[journal, twoside]{IEEEtran}

\ifCLASSINFOpdf
\else
\fi

\usepackage{graphicx}
\usepackage{vector, array, amssymb, amscd, amsthm}
\usepackage[cmex10]{amsmath}
\usepackage{algorithmic}
\usepackage[tight,footnotesize]{subfigure}
\usepackage{stfloats}
\usepackage{url}
\usepackage{cite} 
\usepackage[geometry]{ifsym}
\usepackage{bbm}
\usepackage{algorithm}
\usepackage{algorithmic}
\usepackage{enumerate}
\usepackage{color}
\usepackage{tikz}
\usetikzlibrary{calc,decorations.pathmorphing}

\let\oldFootnote\footnote
\newcommand\nextToken\relax

\renewcommand\footnote[1]{%
    \oldFootnote{#1}\futurelet\nextToken\isFootnote}

\newcommand\isFootnote{%
    \ifx\footnote\nextToken\textsuperscript{,}\fi}

\newtheorem{assumption}{Assumption}

\newtheorem{lemma}{Lemma}
\newtheorem{thm}{Theorem}

\newcommand{\T}{^\top}
\renewcommand{\th}{\text{-th}}

\newcommand{\dt}{{\Delta}}

\newcommand{\To}{T_{o}}

\makeatletter
\newenvironment{proofof}[1]{\par
  \pushQED{\qed}%
  \normalfont \topsep6\p@\@plus6\p@\relax
  \trivlist
  \item[\hskip\labelsep
        \bfseries
    Proof of #1\@addpunct{.}]\ignorespaces
}{%
  \popQED\endtrivlist\@endpefalse
}
\makeatother

\begin{document}

\title{Energy management for building district cooling: a distributed approach to resource sharing}

\author{Fabio Belluschi, Alessandro~Falsone, Daniele~Ioli, Kostas~Margellos, Simone~Garatti and Maria~Prandini
\thanks{Research was supported by the European Commission, H2020, under the project UnCoVerCPS, grant number 643921.}\vspace{-1\baselineskip}
\thanks{F. Belluschi, A. Falsone, D. Ioli, S. Garatti and M. Prandini are with the Dipartimento di Elettronica Informazione e Bioingegneria, Politecnico di Milano,
Piazza Leonardo da Vinci 32, 20133 Milano, Italy, e-mail: \texttt{\{fabio.belluschi, alessandro.falsone, daniele.ioli, simone.garatti, maria.prandini\}@polimi.it}

K. Margellos is with the Department of Engineering Science, University of Oxford, Parks Road,
Oxford, OX1 3PJ, United Kingdom, e-mail: \texttt{kostas.margellos@eng.ox.ac.uk}
}}

\maketitle
\IEEEpeerreviewmaketitle

\begin{abstract}
This paper deals with energy management in a district where multiple buildings can communicate over a time-varying network and aim at optimizing the use of shared resources like storage systems. We focus on building cooling, and propose an iterative, distributed algorithm that accounts for information privacy, since buildings are not required to disclose information about their individual utility functions and constraint sets encoding, e.g., their consumption profiles, and overcomes the communication and computational challenges imposed by centralized management paradigms. Our approach relies on a methodology based on proximal minimization that has recently appeared in the literature. Motivated by the structure of the considered energy management optimization program, we provide a theoretical extension of this novel methodology, that is applicable to any problem that exhibits such structural properties. The efficacy of the resulting energy management algorithm is illustrated by means of a detailed simulation based study, considering different network topologies.
\end{abstract}

\begin{IEEEkeywords}
Energy management, building control, smart grid control, distributed optimization, proximal minimization.
\end{IEEEkeywords}

\section{Introduction} \label{sec:secI}
\IEEEPARstart{O}{ptimal} energy management in buildings has attracted significant attention worldwide, since recent studies \cite{Lausten_2008} have shown that more than 30\% of the total electricity consumption in Europe and in the United States is related to buildings and half of that to climate control. On the same time, a paradigm shift in energy operations has been observed, where energy management shall be performed at a network level, with buildings sharing certain equipment such as storage devices that might be expensive to have at an individual level. Constructing algorithms for optimal energy management in building networks will play a prominent role in the envisioned operational paradigm, allowing demand modulation through intelligent control and coordination of certain appliances, or demand deferrability by appropriate use of the storage devices. To achieve this, not only conventional energy management methods need to be revisited, but also conceptually different control and coordination schemes have to be designed.

Towards this direction, optimization based algorithms have been already successfully applied to the problem of energy management in buildings, due to their ability to handle the multi-objective nature of the problem (e.g., minimize energy costs, maximize building utility), while taking physical and/or technological constraints (e.g., storage limits, comfort constraints) into account. Studies in this direction include, but are not limited to, \cite{Henze_etal_2005,Ma_etal_2009,Siroky_etal_2011,Ma_etal_2012,Oldewurtel_etal_2012,Deng_etal_2013}. Moreover, numerical tools that support these algorithms together with further modeling refinements have been documented in \cite{Sturzenegger_etal_2016}, under the framework of the OptiControl research project \cite{OptiControl}. In most of these works the building modules are modelled by resistance-capacitance circuits, where the temperature set-point of each building zone or room is treated as control input (assuming that a low level controller will track this set-point), set according to the optimization program that encodes the energy management problem. To achieve tractability, both the objective and the constraint functions are chosen to be convex (or some convex approximation is employed) and linearized system dynamics are considered.

Despite the notable research activity in building control and energy management, in all the aforementioned references the underlying algorithm refers to a single building or room. Recently, in \cite{Ioli_etal_2015,Ioli_etal_2016}, a compositional perspective is adopted, allowing for smart-grid control that involves scheduling of multiple building zones, chiller plants, storage devices, etc, interacting with each other, whereas in \cite{Darivianakis_etal_2015} an energy-hub perspective is adopted, investigating the problem of managing a collection of buildings in a cooperative manner. However, the network encoding the interaction among the different modules is considered to be time-invariant, and the problem is solved in a centralized fashion.
In~\cite{ioli_CDC2015}, a hierarchical scheme implementing a decentralized heuristic solution is proposed, which accounts also for the on-off switching of devices.
In \cite{Chang_etal_2013} a decentralized control methodology is applied to a home energy management problem, whereas in \cite{Kranning_etal_2014} a decentralized approach for solving optimal power flow type of problems in transmission networks is proposed, using the alternating direction method of multipliers \cite{Bertsekas_Tsitsiklis_1997,Boyd_etal_2010}. In both cases the underlying  network topology is assumed to be time-invariant. There are two main challenges in all these approaches: First, a time-invariant network does not allow encoding unanticipated events like component and/or communication failures. Second, adopting a centralized or decentralized regulation regime requires all buildings to share information about their consumption patterns, encoded through their individual objective (utility) functions and constraint sets, with each other, or to disclose this information to some central authority. This raises, however, privacy issues, increases the communication requirements and poses computational challenges for large scale problem instances.

In this paper, we overcome these difficulties and deal with the problem of distributed energy management in buildings connected over a (possibly) time-varying network, sharing common resources like a storage unit. To model each individual building in the network we adopt the convex formulation proposed in \cite{Ioli_etal_2015,Ioli_etal_2016}, which constitutes a model of finer granularity compared to \cite{Sturzenegger_etal_2016}. Our distributed methodology is based on the iterative algorithm that has been recently proposed in \cite{Margellos_ACC2016,consensus_paper_2016}, and relies on proximal minimization \cite{Bertsekas_Tsitsiklis_1997}. It is of similar nature with gradient/subgradient algorithms \cite{Nedic_Ozdaglar_2009,Nedic_etal_2010,Chang_etal_2014}, however, it imposes fewer assumptions and it does not require differentiability of the objective function or computation of subgradients. The reader is referred to \cite{consensus_paper_2016} for a detailed comparison with gradient/subgradient algorithms. The building control problem exhibits a particular structure, involving a high number of local optimization variables, and only a few number of global ones that couple the objective functions and the constraint sets of the buildings. We exploit this structure and propose an extension of the algorithm in \cite{Margellos_ACC2016,consensus_paper_2016}, which is applicable to any problem that exhibits such a structure and has the advantage of reducing the exchange of information to the global optimization variables only, thus achieving significant communication savings compared to \cite{Margellos_ACC2016,consensus_paper_2016}. The value taken by the global optimization variables will indeed affect the optimization of the local ones, which, however, is performed locally to each agent. The theoretical guarantees provided in \cite{Margellos_ACC2016,consensus_paper_2016} on the convergence to a minimizer of the centralized counterpart of the problem are shown to still hold in this structured setting.
Note that in \cite{Chang_etal_2014}, a distributed algorithm over a time-varying network is developed, and is applied to a demand side management problem. However, it does not involve models of the same detail as the ones considered here, while from an algorithmic point of view, it is based on subgradient methods as opposed to the proximal minimization approach adopted in this paper, and involves a primal-dual iterative scheme that leads to a higher communication exchange compared to the proposed algorithm that does not require a dual update step.

Our contributions can be summarized as follows: 1) We provide a general formulation of the energy management problem over time-varying networks. 2) We extend the distributed algorithm of \cite{Margellos_ACC2016,consensus_paper_2016} to a certain class of problems, achieving optimality while imposing fewer communication requirements. 3) We apply the proposed algorithm to the building energy management problem, and perform a detailed simulation based study, considering different network topologies.

The rest of the paper unfolds as follows: In Section \ref{sec:secII} we provide details regarding the modeling of a district network, which include buildings and their chillers, and a shared thermal storage. Section \ref{sec:secIII} provides the formulation of the energy management problem over a building network, and includes a detailed exposition and analysis of the proposed distributed algorithm. Section \ref{sec:secIV} provides a simulation based study, whereas Section \ref{sec:secV} concludes the paper and provides some directions for future work. All proofs omitted from the main body of the paper can be found in the Appendix.

\section{Mathematical modeling} \label{sec:secII}
In this section we provide a description of the mathematical models for an individual building, a chiller plant and an energy storage, which constitute the basic modules of the district network for the analysis of Section \ref{sec:secIII}, where the energy management problem of a network of buildings interacting with each other, each of them equipped with its own chiller plant, that share a common storage, is considered.
The models considered in this work are necessarily simplified to facilitate the control purposes of Section \ref{sec:secIII}, but they offer an accurate enough representation of a building from an energy management point of view.
We will focus on energy management over a finite time horizon, divided into $n_t$ time slots, each of them having duration $\dt \in \mathbb{R}$. Therefore, modeling concerns the energy contribution of the building, the chiller plant and the storage per time slot $t$, $t=1,\ldots,n_t$. In the sequel we provide a detailed description for each of them.

\subsection{Building energy request} \label{sec:secIIA}
Consider a building composed of $n_z$ zones. For all $t=1,\ldots,n_t$, $z=1,\ldots,n_z$, let $E_{B,z}(t) \in \mathbb{R}$ be the cooling energy request of building zone $z$ during time slot $t$. Denote then by $E_B(t) = \sum_{z=1}^{n_z} E_{B,z}(t)$ the energy request of the building over the time slot $t$, and by
\begin{align}
	E_B=\sum_{t=1}^{n_t} E_{B}(t) = \sum_{t=1}^{n_t} \sum_{z=1}^{n_z} E_{B,z}(t), \label{eq:total_energy}
\end{align}
the energy request of the building over the entire horizon.
For all $t=1,\ldots,n_t$,, $z=1,\ldots,n_z$, $E_{B,z}(t)$ constitutes of four energy contributions, namely
\begin{align}
	E_{B,z}(t) = E_{\mathrm{walls},z}(t) &+ E_{\mathrm{people},z}(t) \nonumber \\
&+ E_{\mathrm{internal},z}(t) + E_{\mathrm{inertia},z}(t), \label{eq:cooling_energy}
\end{align}
where $E_{\mathrm{walls},z}(t) \in \mathbb{R}$ is the amount of thermal energy exchanged between walls and zone
$z$ over the time slot $t$, $E_{\mathrm{people},z}(t) \in \mathbb{R}$ and $E_{\mathrm{internal},z}(t) \in \mathbb{R}$ is the thermal energy produced by people and by other internal sources of heat in zone $z$, respectively, and $E_{\mathrm{inertia},z}(t) \in \mathbb{R}$ is the energy contribution of the thermal inertia of zone $z$, over the time slot $t$. We next provide a detailed description for each of these terms.

\subsubsection{Walls-zone energy exchange}
Walls separate the building zones from each other, as well as from the outside ambient.
Each wall is divided into vertical layers, referred to as slices, that differ in width and material composition, while each of them is assumed to have uniform density and temperature.
Each internal slice exchanges heat only with nearby slices
through conduction, while boundary slices also exchange heat via convection and thermal radiation through surfaces that are exposed towards either a zone or the outdoor environment.
External surfaces are assumed to be grey and opaque, with equal absorbance and emissivity and with zero transmittance.
Absorbance and emissivity are wavelength-dependent quantities, hence we consider different values for shortwave and longwave radiation \cite{Kim_etal_2013}.

Let $n_w$ denote the number of walls in the building, and assume that each wall is composed of $n_s$ slices. The fact that we assume the same number of slices per wall is to simplify notation, while the case of a different number of slices per wall is straightforward to model. Following \cite{Ioli_etal_2015}, for all $w = 1,\ldots,n_w$, $s = 1,\ldots,n_s$, the temperature $T_{w,s}$ in slice $s$ of wall $w$, evolves according to
\begin{align}
	\dot{T}_{w,s} &= \frac{1}{C_{w,s}} \Big[(h_{w,s}^{s-1} + \bar{h}_{w,s}^{s-1})T_{w,s-1} + (h_{w,s}^{s+1} + \bar{h}_{w,s}^{s+1})T_{w,s+1} \nonumber \\
                    &-(h_{w,s}^{s-1} + \bar{h}_{w,s}^{s-1} + h_{w,s}^{s+1} + \bar{h}_{w,s}^{s+1})T_{w,s} \nonumber \\
					&+ \alpha_{w,s}^S Q^S + \alpha_{w,s}^L Q^L - \varepsilon_{w,s} Q^R(T_{w,s}) + Q^G_{w,s} \Big],
	\label{eq:slice_balance}
\end{align}
where $C_{w,s} \in \mathbb{R}$ denotes the thermal capacity per unit area, and $h_s^{s'}, \bar{h}_s^{s'} \in \mathbb{R}$, with $s' = s-1, s+1$, represent the conductive and the convective heat transfer coefficients between slices $s$ and $s'$, respectively. $Q^S, Q^L \in \mathbb{R}$ denote the incoming shortwave and longwave radiation power per unit area, respectively, while $\alpha_{w,s}^S, \alpha_{w,s}^L \in \mathbb{R}$ are the corresponding absorbance rates for slice $s$ of wall $w$. $Q^R(\cdot): \mathbb{R} \to \mathbb{R}$ is the emitted radiation as a function of the temperature $T_{w,s}$, $\varepsilon_{w,s} < 1$ is the emissivity, and $Q^G_{w,s} \in \mathbb{R}$ is the thermal power generation inside slice $s$ of wall $w$.

In \eqref{eq:slice_balance}, for $s=1$, $s=n_s$, the quantities $T_{w,0}$ and $T_{w,n_s+1}$ appear, respectively. They denote the temperature of either a zone of the building or the ambient temperature, according to whether the boundary slice $1$ (similarly for $n_s$) of wall $w$ is at the border with some building zone or with the outdoor environment. It should be noted that, for all $w=1,\ldots,n_w$, we have that
$h_{w,1}^0 = h_{w,n_s}^{n_s+1} = 0$ since there is no thermal conduction on slices that are the walls' boundary surfaces, and
$\bar{h}_{w,s}^{s-1} = 0$, for all $s>1$, $\bar{h}_{w,s}^{s+1} = 0$, for all $s<n_s$, and $\alpha_{w,s}^S = \alpha_{w,s}^L = \varepsilon_{w,s} = 0$, for all $1<s<n_s$,
since there is no thermal convection nor radiation in between slices.
Since each wall is assumed to be a grey body, the power $Q^R(T_{w,s})$ radiated from slice $s$ of wall $w$, is given by $Q^R(T_{w,s}) = \sigma T_{w,s}^4$, where $\sigma$ is the Stefan-Boltzmann constant (see \cite{Ioli_etal_2015} and references therein). This expression is approximately linear around the mean operating temperature $\bar{T}_{w,s}$ of slice $s$ of wall $w$, and hence the aforementioned nonlinear expression can be approximated by the following equation, which is linear in $T_{w,s}$.
\begin{align}
	Q^R(T_{w,s}) = 4 \sigma \bar{T}_{w,s}^3 T_{w,s} - 3 \sigma \bar{T}_{w,s}^4.
	\label{eq:stefan_boltzmann_lin}
\end{align}

For $w=1,\ldots,n_w$, denote by $T_w = [T_{w,1} \cdots T_{w,n_s}]\T \in \mathbb{R}^{n_s}$ the vector that contains the temperatures of all slices of wall $w$. For $z= 1,\ldots,n_z$, let $\widetilde{T}_z \in \mathbb{R}$ denote the temperature of zone $z$, and $\widetilde{T} = [\widetilde{T}_{1} \cdots \widetilde{T}_{n_z}]\T  \in \mathbb{R}^{n_z}$ be a vector including all zone temperatures, Moreover, let $T_o \in \mathbb{R}$ denote the ambient temperature. By \eqref{eq:slice_balance}, and noticing that for each $w=1,\ldots,n_w$, some elements of $\widetilde{T}$ and/or $T_o$ correspond to the terms $T_{w,0}$, $T_{w,n_s+1}$, $T_w$ depends on $\widetilde{T}$ and $T_o$, and its evolution is given by
\begin{align}
	\dot{T}_w = A_wT_w + B_w\widetilde{T} + F_wd,
	\label{eq:wall_dynamics}
\end{align}
where $d = [\To \; Q^S \; Q^L \; 1]\T$ acts as a disturbance vector, collecting the ambient temperature $\To$, and the incoming shortwave $Q^S$ and longwave $Q^L$ radiation, while the constant $1$ is introduced to account for the constant terms in \eqref{eq:slice_balance} and \eqref{eq:stefan_boltzmann_lin}.
Matrices $A_w$, $B_w$ and $F_w$ are of appropriate dimension and their elements depend on the constant parameters in \eqref{eq:slice_balance}.
Letting $T = [T_1\T \cdots T_{n_w}\T]\T \in \mathbb{R}^{n_s n_w}$ be a vector including the temperatures of all slices and all walls, by \eqref{eq:wall_dynamics} we have that
\begin{align}
	\dot{T} = AT + B\widetilde{T} + Fd,
	\label{eq:building_dynamics}
\end{align}
where $A$ is a block-diagonal matrix with $A_w$ in its $w\th$ block, $B = [B_1\T \cdots B_{n_w}\T]\T$ and $F = [F_1\T \cdots F_{n_w}\T]\T$.

For all $z=1,\ldots,n_z$, let $W_z \subseteq \{1,\ldots,m\}$ denote the set of indices that correspond to walls that are adjacent to zone $z$. For all $z=1,\ldots,n_z$, the thermal power that is transferred to zone $z$ from its adjacent walls is given by $Q_{\mathrm{walls},z} = \sum_{w \in W_z} S_w h_{w,s}^{s'} (T_{w,s} - \widetilde{T}_z)$, where $S_w$ is the surface area of wall $w$. The pair $(s,s')$ is either $(1,0)$ or $(n_s,n_s+1)$, according to which slice (i.e., $s=1$ or $s=n_s$) is at the border of wall $w \in W_z$ with zone $z$.
Defining $Q_\mathrm{walls} = [Q_{\mathrm{walls},1} \cdots Q_{\mathrm{walls},n_z}]\T \in \mathbb{R}^{n_z}$, we can thus represent $Q_\mathrm{walls}$ as a function of $T$ and $\widetilde{T}$, i.e.,
\begin{align}
	Q_\mathrm{walls} = CT + D\widetilde{T},
	\label{eq:building_output}
\end{align}
where $C$ and $D$ are matrices of appropriate dimension.

Equations \eqref{eq:building_dynamics} and \eqref{eq:building_output} form a linear dynamical system with state $T$, input $\widetilde{T}$, disturbance $d$, and output $Q_\mathrm{walls}$.
The obtained system, though linear, can be quite large. However, following \cite{Kim_Brown_2012},
its order can be greatly reduced by applying model reduction techniques; this approach was adopted in the case study of Section \ref{sec:secIV}.

Since the energy management algorithm of Section \ref{sec:secIII} is developed in discrete time, we can discretize
\eqref{eq:building_dynamics}-\eqref{eq:building_output} to obtain a discrete time linear dynamical system. For the discretization process it was assumed that $\widetilde{T}$ and $d$ vary linearly within each time slot of duration $\dt$.
Let then $Q_{\mathrm{walls}}(t \dt)$, $t=1,\ldots,n_t$, be the thermal power calculated by meas of the discretized model, where, with a slight abuse of notation, we use the same symbol with the continuous time model (see \cite{Ioli_etal_2016} for more details on the discretization process).
Due to the fact that the elements of $Q_{\mathrm{walls}}(\cdot)$ were assumed to vary linearly within each time slot, the thermal energy $E_{\mathrm{walls},z}(t)$ transferred by all walls to zone $z$ over the time slot $t$, can be computed from the thermal power as
\begin{align}
E_{\mathrm{walls},z}(t) = \frac{\dt}{2} (Q_{\mathrm{walls},z}((t-1)\dt) + Q_{\mathrm{walls},z}(t \dt)), \label{eq:energy_wz}
\end{align}
for all $z=1,\ldots,n_z$, $t=1,\ldots,n_t$. Note that after discretizing \eqref{eq:building_dynamics}-\eqref{eq:building_output}, $Q_{\mathrm{walls},z}(t \dt)$, and hence also $E_{\mathrm{walls},z}(t)$, depends on an affine fashion on the temperatures of zone $z$ ut to time $t$, i.e., $\widetilde{T}_z(t'\dt)$ for all $t' \leq t$.

\subsubsection{People energy contribution}
People produce heat and, in buildings with high occupancy, their contribution to the total thermal energy generation is significant.
Let $n_{p,z}$ denote the number of occupants in zone $z$, and denote by $Q_{\mathrm{people},z} \in \mathbb{R}$ the thermal power produced by them at zone temperature $\widetilde{T}_z$, $z=1,\ldots,n_z$. Following \cite{Butcher_2006}, $Q_{\mathrm{people},z}$ is given as a product of $n_{p,z}$ and a quadratic function of $\widetilde{T}_z$. However, this function is approximately linear for the sensible range of operating temperatures that are of interest for our analysis. Therefore, without introducing a significant modeling error, for each $z=1,\ldots,n_z$, we can linearize it around some comfort temperature of zone $z$, thus obtaining
\begin{align}
Q_{\mathrm{people},z} = n_{p,z} (p_{1,z} \widetilde{T}_z + p_{0,z} ),
		\label{eq:people_heat_rate_lin}
\end{align}
where $p_{0,z}, p_{1,z} \in \mathbb{R}$ are constants that are different per zone $z$, since they depend on the chosen comfort temperature of zone $z$.

Note that $n_{p,z}$ and $\widetilde{T}_z$ depend on time continuously. Moreover, following \cite{Borghesan_etal_2013}, $n_{p,z}$ can be approximated as a linear function of time, while $\widetilde{T}_z$, $z=1,\ldots,n_z$, has been already assumed to evolve linearly within each time slot. Therefore, for each $t = 1,\ldots,n_t$, $z=1,\ldots,n_z$, we can then integrate the people thermal power $Q_{\mathrm{people},z}$ in \eqref{eq:people_heat_rate_lin} analytically over $[(t-1)\dt, t\dt]$ to obtain the energy contribution $E_{\mathrm{people},z}(t)$ due to people occupancy at zone $z$ over the time slot $t$, i.e.,
\begin{align}
E_{\mathrm{people},z}(t) &= q_{2,z} (t) \widetilde{T}_z(t\dt) \nonumber \\
&+ q_{1,z}(t) \widetilde{T}_z((t-1)\dt) + q_{0,z}(t),
\end{align}
where the coefficients $q_{0,z}(t), q_{1,z}(t), q_{2,z}(t) \in \mathbb{R}$ are different per zone $z$ and time slot $t$, since they depend on the values of $n_{p,z}$ at $(t-1)\dt$ and $t\dt$; further details can be found in \cite{Ioli_etal_2015}.

\subsubsection{Other internal energy contributions}
There are many other types of heat sources that may affect the internal thermal energy of a building, e.g., internal lighting, electrical equipment, daylight radiation through windows, etc. For each $z=1,\ldots,n_z$, the overall thermal energy of zone $z$ due to these sources can be computed as
\begin{equation}
	Q_{\mathrm{internal},z} = \alpha_z Q^S + \beta_z \max(n_{p,z},0) + \gamma_z,
	\label{eq:windows_heat_rate}
\end{equation}
where $\alpha_z \in \mathbb{R}$ is a coefficient that takes into account the mean absorbance coefficient of zone $z$, the transmittance coefficients of the windows and their areas, the sun view and the shading factors, and the radiation incidence angle.
The last two terms in \eqref{eq:windows_heat_rate} encode the thermal energy contribution to zone $z$, $z=1,\ldots,n_z$, due to internal lightening and electrical equipment. In particular, it is composed of a constant term $\gamma_z \in \mathbb{R}$, and an additional contribution
$\beta_z \in \mathbb{R}$ when people are present, i.e., when $n_{p,z} > 0$.
Note that $Q_{\mathrm{internal},z}$ does not depend on the longwave radiation $Q^L$, due to the fact that the windows are usually shielded against it. Note that $\beta_z$ and $\gamma_z$ are time independent, however, $n_{p,z}$, $Q^S$ and $\alpha_z$ depend on time continuously. Therefore, we first discretize the thermal power $Q_{\mathrm{internal},z}$ in \eqref{eq:windows_heat_rate}, and then integrate it to obtain the energy term $E_{\mathrm{internal},z}(t)$ within each time slot $t$, $t=1,\ldots,n_t$.

\subsubsection{Energy contribution due to zone inertia}
To lower the temperature of a zone we need to draw energy from the zone itself. Following \cite{Ioli_etal_2015}, for all $t=1,\ldots,n_t$, $z=1,\ldots,n_z$, this inertial contribution of zone $z$ to the overall thermal energy \eqref{eq:cooling_energy} in the building over the time slot $t$ , can be expressed as
\begin{equation}
	E_{\mathrm{inertia},z}(t) = -\bar{C}_z (\widetilde{T}_z(t\dt) - \widetilde{T}_z((t-1)\dt)),
	\label{eq:zones_energy}
\end{equation}
where $\bar{C}_z \in \mathbb{R}$ is the equivalent heat capacity of the $z\th$ zone.

\subsection{Chiller plant} \label{sec:secIIB}
A chiller plant converts electric energy into cooling energy. The cooling energy is then transferred to the building via, e.g., the chilled water circuit. Denote by $E_{\mathrm{chiller},e}(t) \in \mathbb{R}$ the electric energy absorbed by the chiller to provide cooling energy $E_{\mathrm{chiller},c} \in \mathbb{R}$ over a time slot of duration $t$, $t=1,\ldots,n_t$.
Following the Ng-Gordon model of \cite{Gordon_Ng_2000}, which is based on entropy and energy balance equations, the electric and the cooling energy can be related by
\small{
\begin{align}
 &E_{\mathrm{chiller},e}(t) = \frac{1}{T_{cw}(t \dt)-\frac{\alpha_4}{\dt}E_{\mathrm{chiller},c}(t)} \Big ( \alpha_1 T_o(t\dt) T_{cw}(t\dt) \dt \nonumber \\ &~~~+ \alpha_2(T_o(t\dt)-T_{cw}(t\dt))\dt + \alpha_3T_o(t\dt)E_{\mathrm{chiller},c}(t) \Big ) \nonumber \\ &~~~- E_{\mathrm{chiller},c}(t). \label{eq:chiller}
\end{align}
} \normalsize
$T_o(t\dt)$ denotes the ambient temperature and $T_{cw}(t\dt)$ the temperature of the cooling water at time slot $t$, $t=1,\ldots,n_t$. The latter is typically regulated by low level controllers so that it is maintained at some prescribed optimal operational value.
Coefficients $\alpha_1, \alpha_2, \alpha_3, \alpha_4 \in \mathbb{R}$ characterize the chiller performance and, depending on their values, we can have different efficiency curves as given by the so called coefficient of performance (COP). The COP is the ratio between the produced cooling energy and the corresponding electrical energy consumption. As such, the larger is the COP, the more efficient the chiller is.

To facilitate the optimization based algorithm of Section \ref{sec:secIIIB}, we employ the following convex approximation of \eqref{eq:chiller}, which results in a biquadratic relationship between the electrical and the cooling energy of the chiller.
\begin{align}
    E_{\mathrm{chiller},e}(t) &= c_2(T_o(t\dt)) E^4_{\mathrm{chiller},c}(t) \nonumber \\ &+ c_1(T_o(t\dt)) E^2_{\mathrm{chiller},c}(t) + c_0(T_o(t\dt)), \label{eq:chiller_biquad}
\end{align}
where, for each $t=1,\ldots,n_t$, the functions $c_0(\cdot), c_1(\cdot), c_2(\cdot): \mathbb{R} \to \mathbb{R}$ depend on the ambient temperature $T_o(t\dt)$. The approximate relationship in \eqref{eq:chiller_biquad} is determined from \eqref{eq:chiller} using weighted least squares to best fit the most relevant points, i.e, those that correspond to zero energy request and to the maximum COP values.
The quality of the biquadratic approximation in \eqref{eq:chiller_biquad}, compared to the full model in \eqref{eq:chiller}, is illustrated in Figure \ref{fig:chiller}.

\begin{figure}[ht]
 \centering
 \includegraphics[width = 0.9\columnwidth]{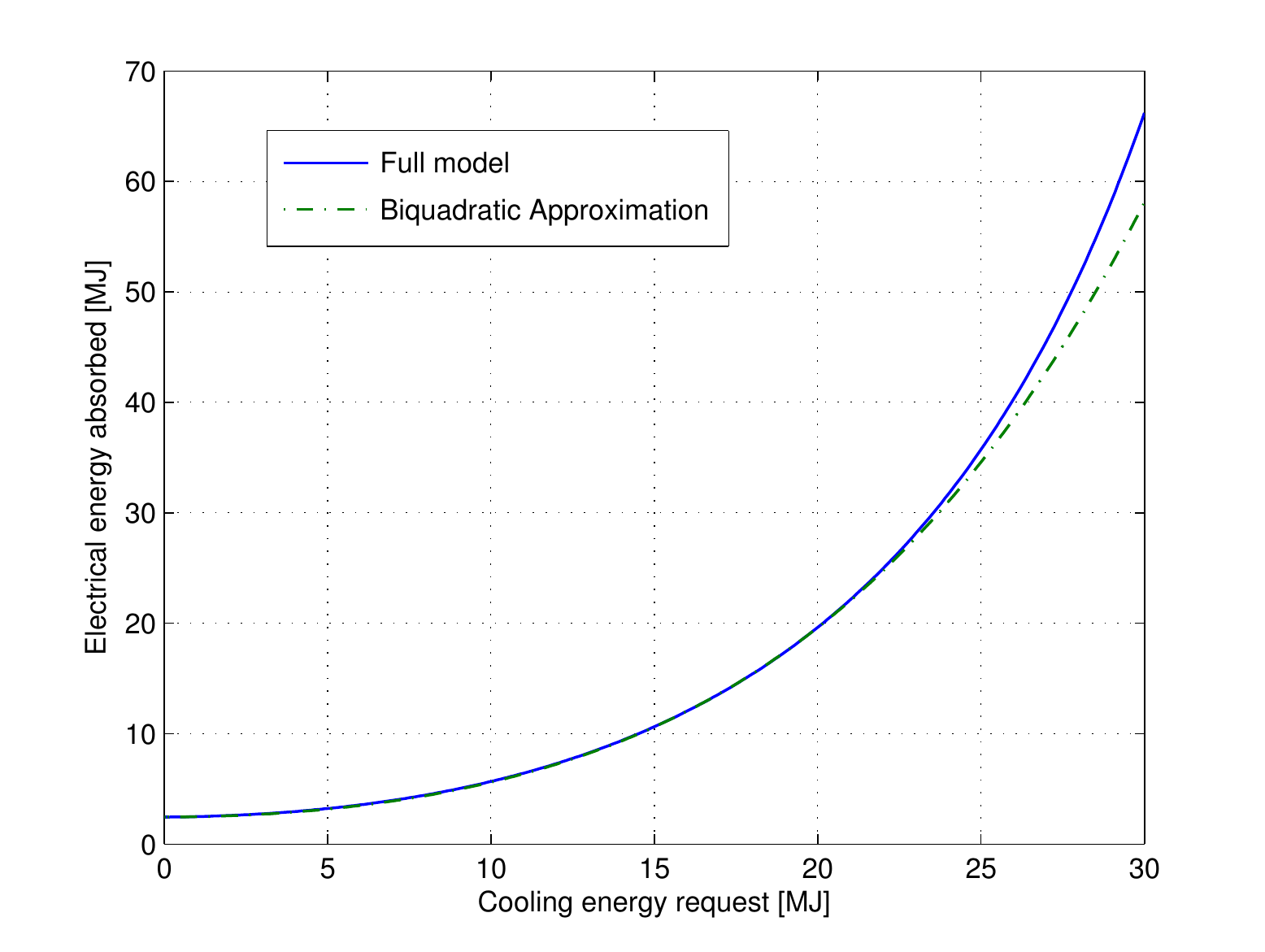}
 \caption{Electrical energy consumption as a function of the cooling energy request: Full model (solid line) and its biquadratic approximation (dashed-dotted line).}
 \label{fig:chiller}
\end{figure}

\subsection{Energy storage} \label{sec:secIIC}
Thermal energy storage is becoming widely used since it represents the most effective way, and often the only way, to take advantage of renewable energy sources. There are many different technical solutions to store thermal energy; the most widely used are fluid tanks and phase changing materials. In buildings, and more generally in a smart grid context, they can be used as energy buffers for unbinding energy production from energy consumption. In particular, thermal energy storage for cooling energy can shift the production of cooling energy to off-peak hours of electrical energy consumption, make chillers operate in high-efficiency conditions, and smoothen peaks of electrical energy request with benefits both for power production and distribution network systems.

From an energy oriented perspective we will model a thermal energy storage employing black box system identification techniques, that use the energy exchange (drawn or inserted) as input and the thermal energy stored as output. This way of modeling does not consider the way energy is stored or provided. To this end, a first order AutoRegressive eXogenous (ARX) system is considered:
\begin{align}
E_{\mathrm{storage}}(t+1) = aE_{\mathrm{storage}}(t) - \sum_{i=1}^m e_s^i(t), \label{eq:storage}
\end{align}
where $E_{\mathrm{storage}}(t) \in \mathbb{R}$ is the amount of cooling energy stored. In view of the multi-building problem considered in the next section we assume that the storage device is shared among $m$ buildings, and denote by $e_s^i(t) \in \mathbb{R}$ the cooling energy exchanged ($e_s^i(t)>0$ if the storage is discharged, and  $e_s^i(t)<0$ if it is charged), with building $i$, $i=1,\ldots,m$, in time slot $t$, $t=1,\ldots,n_t$.  The coefficient $a \in (0,1)$ is introduced to model energy losses. Note that \eqref{eq:storage} can be thought of as a discrete time integrator, where the stored energy $E_{\mathrm{storage}}(t+1)$ at time $t+1$ is computed by accumulating the cooling energy exchanged with all buildings up to time $t$, i.e., $\sum_{i=1}^m e_s^i(t')$ for all $t' \leq t$.

\section{Energy management of a building network} \label{sec:secIII}
\subsection{Problem statement} \label{sec:secIIIA}
Consider a network of $m$ buildings, each of them equipped with a different chiller plant, that share a common storage so as to avoid usage inefficiencies and increase the return on investment of the storage resource, which might be expensive to afford at an individual level. To this end, append to all quantities introduced in the previous section the superscript $i$, to denote that they correspond to building $i$, $i=1,\ldots,m$, e.g., $E_{\mathrm{chiller},e}^i(t)$ denotes the cooling energy of the chiller at building $i$ at time slot $t$, $\widetilde{T}^i$ denotes the vector of zone temperatures at building $i$, etc.

For each $i=1,\ldots,m$, $t=1,\ldots,n_t$, the electric energy request of building $i$ over the time slot $t$ is given by the chiller electric energy request $E_{\mathrm{chiller},e}^i(t)$. The latter is in turn related to the cooling energy exchange terms via \eqref{eq:chiller_biquad}.
Our objective is to minimize the total electric energy cost for the $m$ building network, across a horizon of $n_t$ steps. To achieve this, for each building $i$, $i=1,\ldots,m$, we will schedule the zone temperature set-points $\widetilde{T}^i(t \dt)$, the energy exchange $e_s^i(t)$ with the storage, and the initial conditions for the temperature vector
$T^i(\dt)$ (including slice and wall temperatures), and the storage level $E_{\mathrm{storage}}(1)$. Therefore, we seek to solve the following minimization problem:
\begin{align}
\min_{\begin{subarray}{c} \big \{ \big \{ \widetilde{T}^i(t \dt) \in \mathbb{R}^{n_z}, e_s^i(t) \in \mathbb{R} \big \}_{t=1}^{n_t} \big \}_{i=1}^m, \\ \big \{ T^i(\dt) \in \mathbb{R}^{n_s n_w}\big \}_{i=1}^m, E_{\mathrm{storage}}(1) \in \mathbb{R} \end{subarray}}
\sum_{i=1}^m \sum_{t=1}^{n_t} \psi^i(t) E_{\mathrm{chiller},e}^i(t), \label{eq:P_obj_full}
\end{align}
where $\psi^i(t) \in \mathbb{R}$ is the electric energy price for building $i$, $i=1,\ldots,m$, over the time slot $t$, $t=1,\ldots,n_t$. This minimization is subject to the following constraints.
\begin{enumerate}
\item Electric energy request: For each $t=1,\ldots,n_t$, the electric energy request $E_{\mathrm{chiller},e}^i(t)$ of building $i$, $i=1,\ldots,m$, is given by \eqref{eq:chiller_biquad} as a function of the chiller cooling energy request $E_{\mathrm{chiller},c}^i(t)$. The latter denotes the net cooling energy request, and is given by the difference between the total energy requested by the building minus the energy exchanged with the storage, i.e.,
    \begin{align}
    E_{\mathrm{chiller},c}^i(t) = E_{B}^i(t) - e_s^i(t), \label{eq:ch_energyReq}
    \end{align}
    where $E_{B}^i(t)$ is as shown in \eqref{eq:total_energy}, and $e_s^i(t)$ is the energy exchange between building $i$ and the storage (see Section \ref{sec:secIIC}).
\item Electric energy limits: For each $i=1,\ldots,m$, $t=1,\ldots,n_t$, the electric energy drawn from the network is limited to $E_{\max}^i \in \mathbb{R}$, as an effect of the chiller unit size and maximum capability, thus giving rise to
    \begin{align}
    E_{\mathrm{chiller},e}(t) \leq E_{\max}^i. \label{eq:ch_elecEreq}
    \end{align}
\item Cooling energy limits: For each $i=1,\ldots,m$, $t=1,\ldots,n_t$, the cooling energy request $E_B^i(t)$ of building $i$ over time-slot $t$, as given by \eqref{eq:total_energy}, is non-negative, i.e.,
    \begin{align}
    E_{B}^i(t) \geq 0, \label{eq:ch_coolEreq}
    \end{align}
\item Comfort constraints: For each $i=1,\ldots,m$, $t=1,\ldots,n_t$, the zone temperature set-points is within certain limits, i.e.,
    \begin{align}
    \widetilde{T}^i(t \dt) \in [\widetilde{T}^i_{\min}(t \dt),~ \widetilde{T}^i_{\max}(t \dt)], \label{eq:comfort_con}
    \end{align}
    where $\widetilde{T}^i_{\min}(t \dt) \in \mathbb{R}^{n_z}$, $\widetilde{T}^i_{\max}(t \dt) \in \mathbb{R}^{n_z}$ denote the minimum and maximum, respectively, temperature limits, so that comfort is maintained. These limits may differ according to the type of each building.
\item Storage energy limits: For each $t=1,\ldots,n_t$, the amount of cooling energy stored at time slot $t$ should be non-negative and within a prescribed energy storage limit $E_{s,
    \max} \in \mathbb{R}$, i.e.,
    \begin{align}
    E_{\mathrm{storage}}(t) \in [0,~ E_{s,\max}]. \label{eq:storage_Econ}
    \end{align}
\item Storage energy exchange limits: For each $i=1,\ldots,m$, $t=1,\ldots,n_t$, the energy exchanged with the storage is subject to
    \begin{align}
    e_s^i(t) \in [-e_{s,\max}^i,~ e_{s,\max}^i], \label{eq:storage_Pcon}
    \end{align}
    where $e_{s,\max}^i \in \mathbb{R}$ denotes the maximum value of energy that can be exchanged with the storage for building $i$. Notice that we use symmetric limits for positive and negative energy exchanges.
\item Final value constraints: For each $i=1,\ldots,m$, the zone temperature, and the wall-slice temperature, at the beginning and at the end of the planning horizon should be equal, i.e.,
    \begin{align}
    \widetilde{T}^i(n_t \dt) &= \widetilde{T}^i(\dt), \nonumber \\
    T^i(n_t \dt) &= T^i(\dt). \label{eq:init_T}
    \end{align}
    To ensure that at the end of the horizon the storage is sufficiently charged we impose the constraint
    \begin{align}
    E_{\mathrm{storage}}(n_t) \geq E_{\mathrm{storage}}(1), \label{eq:init_storage}
    \end{align}
    where we optimize with respect to $E_{\mathrm{storage}}(1)$ (see \eqref{eq:P_obj}). Constraint \eqref{eq:init_storage} is of particular importance in case of a receding horizon implementation of the proposed scheme.
\end{enumerate}
Note that, even though it is not shown explicitly to ease notation, $E_{\mathrm{chiller},e}^i(t)$ in \eqref{eq:P_obj_full} is a function of the decision variables $\big \{ \big \{ \widetilde{T}^i(t \dt) , e_s^i(t) \big \}_{t=1}^{n_t} \big \}_{i=1}^m, \big \{ T^i(\dt) \big \}_{i=1}^m, E_{\mathrm{storage}}(1)$; this can be verified by tracing the representation of $E_{\mathrm{chiller},e}^i(t)$ via \eqref{eq:chiller_biquad}, \eqref{eq:ch_energyReq}, \eqref{eq:total_energy}, where the energy terms in the latter equation depend on the decision variables according to the analysis of Section \ref{sec:secII}.

For each $i=1,\ldots,m$, denote by
\begin{align}
u_i = \big [ \widetilde{T}^i(\dt), \ldots, \widetilde{T}^i(n_t \dt), T^i(\dt) \big ]^\top \in \mathbb{R}^{n_t n_z + n_sn_w}, \label{eq:var_u}
\end{align}
all temperature related decision variables that correspond to building $i$. Let $\bar{e}_s^i = \big [ e_s^i(1), \ldots, e_s^i(n_t), E_{\mathrm{storage}}(1) \big ]^\top \in \mathbb{R}^{n_t+1}$, and denote by
\begin{align}
x = \big [ \bar{e}_s^1, \ldots, \bar{e}_s^m \big ]^\top \in \mathbb{R}^{m (n_t+1)}, \label{eq:var_x}
\end{align}
the vector including all decision variables related to the energy exchange between the buildings and the storage, and the initial energy storage value. Note that $u_i$ is indexed by $i$, $i=1,\ldots,m$, and can be thus thought of as a local decision vector related to the comfort and actuation constraints of each chiller plant, that can be enforced locally. On the other hand, $x$ is treated as a global decision vector which is related to the energy exchange of the building network with the common storage device. Under the variable assignment in
\eqref{eq:var_u}, \eqref{eq:var_x}, the energy management in \eqref{eq:P_obj_full}-\eqref{eq:init_storage} can be represented in a more compact notation by
\begin{align}
\mathcal{P}:~ & \min_{x \in \mathbb{R}^n, \{u_i \in \mathbb{R}^{n_i} \}_{i=1}^m} \sum_{i=1}^m f_i(x,u_i) \label{eq:P_obj} \\
&\text{subject to } \nonumber \\
&(x,u_i) \in V_i, \text{ for all } i=1,\ldots,m, \label{eq:P_con}
\end{align}
where $f_i(\cdot,\cdot): \mathbb{R}^n \times \mathbb{R}^{n_i} \to \mathbb{R}$ and $V_i \subseteq \mathbb{R}^{n+n_i}$, for all $i=1,\ldots,m$.
Note that $x$ couples the individual decision vectors $u_i$ via the objective  in \eqref{eq:P_obj} and the constraints in \eqref{eq:P_con}.

\subsection{Distributed algorithm} \label{sec:secIIIB}

In this section we will occasionally refer to buildings as agents. We provide a distributed iterative procedure to solve $\mathcal{P}$, where, at every iteration, each agent $i$ solves an appropriate local optimization problem and then exchanges information with other agents only regarding the temporarily obtained value for the common decision vector $x$. In this way, one can account for information privacy, because agents are not required to share the objective function $f_i$, the constraint set $V_i$, and their local decision vector $u_i$, $i=1,\ldots,m$. In a building energy management context, this specifically means that each building does not need to reveal constraints that are related to its local consumption patterns or to occupants' preferences, nor to reveal information about its individual utility function, which may constitute private information in case buildings participate in a demand response program. Moreover, even though all the necessary information could be exchanged, solving $\mathcal{P}$ in a centralized fashion may result computationally intensive
and our distributed algorithm is also a means to alleviate this issue.

The pseudo-code of the proposed distributed procedure is given in Algorithm \ref{alg:Alg1}. In the remainder of this subsection we provide some explanations of the algorithm steps, whereas in Section \ref{sec:secIIIC} we show that, under certain structural and communication assumptions, it converges, and agents reach consensus to a common value for the global decision vector $x$ that, together with the converged values for the local decision vectors $u_i$, $i=1,\ldots,m$, forms an optimal solution of $\mathcal{P}$ (note that $\mathcal{P}$  does not necessarily admit a unique solution).

\begin{algorithm}[t]
\caption{Distributed algorithm}
\begin{algorithmic}[1]
\STATE \textbf{Initialization} \\
\STATE ~~$k=0$. \\
\STATE ~~Consider $x_i(0)$, $u_i(0)$, \\
~~such that $(x_i(0),u_i(0)) \in V_i$ for all $i=1,\ldots,m$. \\
\STATE \textbf{For $i=1,\ldots,m$ repeat until convergence} \\
\STATE ~~$\bar{x}_i(k) = \sum_{j=1}^m a_j^i(k) x_j(k)$. \\
\STATE ~~$(x_i(k+1),u_i(k+1))$ \\
~~~~~~$\in \arg \min_{(x_i,u_i) \in V_i} f_i(x_i,u_i) + \frac{1}{2 c(k)} ||\bar{x}_i(k) - x_i||^2$. \\
\STATE ~~$k \leftarrow k+1$.
\end{algorithmic}
\label{alg:Alg1}
\end{algorithm}

Initially, each agent $i$, $i=1,\ldots,m$, starts with some tentative values $u_i(0)$ and $x_i(0)$ for its local decision vector and the global decision vector, respectively. The latter constitutes an estimate of agent $i$ (this justifies the subscript $i$ in $x_i$) of what the value of the global decision vector might be. Those tentative values are chosen arbitrarily from the set of feasible solutions, i.e., $(x_i(0),u_i(0)) \in V_i$ (step 3).
One sensible choice for $(x_i(0),u_i(0))$ is to set it such that $(x_i(0),u_i(0)) \in \arg \min_{(x_i,u_i) \in V_i, } f_i(x_i,u_i)$.
At iteration $k+1$, each agent $i$ constructs a weighted average $\bar{x}_i(k)$ of the solutions $x_j(k),~j=1,\ldots,m$ communicated by the other agents and its own one (step 5). Coefficient $a_j^i(k) \geq 0$, indicates how agent $i$ weights the solution received by agent $j$ at iteration $k$, and $a_j^i(k) = 0$ encodes the fact that agent $i$ does not receive any information from agent $j$ at iteration $k$ (i.e. the communication link between agents $i$ and $j$ is not active at iteration $k$). Agent $i$ solves then a local minimization problem, seeking the optimal solution pair $(x_i,u_i)$ within $V_i$ that minimizes a performance criterion, which is defined as a linear combination of the local objective function $f_i(x_i,u_i)$ and a quadratic term\footnote{Throughout the paper, $||\cdot||$ denotes Euclidean norm.}, penalizing the difference from $\bar{x}_i(k)$ (step 6). The relative importance of these two terms is dictated by $c(k) > 0$. Since multiple minimizers may exist, we assume that at every iteration the same deterministic tie-break rule (as e.g. that implemented by a deterministic numerical solver) is used.

Algorithm \ref{alg:Alg1} is closely related to the distributed methodology that has been recently proposed in \cite{Margellos_ACC2016,consensus_paper_2016}. However, in Algorithm \ref{alg:Alg1} neighboring agents need to exchange at every iteration their tentative estimates for the value of the global decision vector
only, while the distributed algorithm in \cite{Margellos_ACC2016,consensus_paper_2016} requires to exchange both the global and the local decision vectors. When the dimension of the local decision vector is high compared to the global one, as it is the case in the building energy management problem, this would unnecessarily increase the amount of information that needs to be exchanged. Algorithm \ref{alg:Alg1} alleviates this issue by extending the approach of \cite{Margellos_ACC2016,consensus_paper_2016} to exploit the particular structure of $\mathcal{P}$, where the objective functions and the constraint sets are coupled only by means of $x$.

\subsection{Algorithm analysis} \label{sec:secIIIC}
In this section we study the convergence properties of Algorithm \ref{alg:Alg1}. To this end, we first impose certain assumptions, that need to be satisfied for Algorithm \ref{alg:Alg1} to converge, and then provide a convergence proof. The implications of these assumptions on the energy management problem under study are discussed in Section \ref{sec:secIIID}.

\subsubsection{Assumptions} \label{sec:secIIIC1}

The following structural assumptions are in order.
\begin{assumption} \label{ass:Convex}
 For all $i=1,\ldots,m$, the function $f_i(\cdot,\cdot): \mathbb{R}^n \times \mathbb{R}^{n_i} \to \mathbb{R}$ is jointly convex with respect to its arguments.
 Moreover, for all $i=1,\ldots,m$, $f_i(\cdot,\cdot): \mathbb{R}^n \times \mathbb{R}^{n_i} \to \mathbb{R}$ is jointly Lipschitz continuous with respect to its arguments.
\end{assumption}

Note that under Assumption \ref{ass:Convex}, and due to the presence of the quadratic penalty term , the objective function in the optimization problem at step 6 of Algorithm \ref{alg:Alg1} is strictly convex with respect to $x_i$. Therefore, a unique solution for $x_i$ is admitted; this is not the case for $u_i$.

\begin{assumption} \label{ass:Compact}
For all $i=1,\ldots,m$, the set $V_i \subseteq \mathbb{R}^{n+n_i}$ is compact and convex. Moreover,
$\bigcap_{i=1}^m V_i$ has non-empty interior.
\end{assumption}
For all $i=1,\ldots,m$, for any $x \in \mathbb{R}^n$, consider the set
\begin{align}
U_i(x) = \big \{ u_i \in \mathbb{R}^{n_i}:~ (x,u_i) \in V_i \big \}. \label{eq:con_u}
\end{align}
Moreover, for all $i=1,\ldots,m$, consider the projection of $V_i$ on the $x$ domain, i.e.,
\begin{align}
X_i = \big \{ x \in \mathbb{R}^{n}:~ \exists u_i \in \mathbb{R}^{n_i} \text{ such that } (x,u_i) \in V_i \big \}. \label{eq:con_x}
\end{align}
A direct consequence of the first part of Assumption \ref{ass:Compact} is that, for all $i=1,\ldots,m$, $X_i$ and $U_i(x)$ for any $x \in X_i$ are all compact and convex. By the second part of Assumption \ref{ass:Compact}, we also have that $\bigcap_{i=1}^m X_i$, and hence also $X_i$, $i=1,\ldots,m$, has a non-empty interior. Moreover, $U_i(x)$ is non-empty for any $x \in X_i$, $i=1,\ldots,m$.

The fact that $V_i$ is both convex and compact implies that the set-valued mapping $U_i(\cdot)$ is continuous on $X_i$, see \cite{Aubin_1991}. In the following assumption we further require that $U_i(\cdot)$ is Lipschitz continuous.
\begin{assumption} \label{ass:Lipschitz}
For all $i=1,\ldots,m$, the set-valued mapping $U_i(\cdot): X_i \rightrightarrows \mathbb{R}^{n_i}$ is Lipschitz continuous, i.e., there exists $L_i \in \mathbb{R}$, $L_i >0$, such that
\begin{align}
d_H(U_i(x),U_i(x')) \leq L_i ||x - x'||, \text{ for all } x, x' \in X_i,\label{eq:Lipschitz_set}
\end{align}
where
\begin{align}
&d_H(U_i(x),U_i(x')) \nonumber \\
&=\sup_{u_i \in \mathbb{R}^{n_i}} \big | \min_{v_i \in U_i(x)} ||u_i - v_i|| -  \min_{v'_i \in U_i(x')} ||u_i - v'_i|| \big |, \label{eq:hausdorff}
\end{align}
denotes the Pompeiu-Hausdorff distance (see p. 272 in \cite{Wets_2003}) between the sets $U_i(x)$ and $U_i(x')$.
\end{assumption}

We also impose the following technical assumptions
\begin{assumption} \label{ass:ConvCoef}
$\{c(k)\}_{k \geq 0}$ is a non-increasing sequence with $c(k) > 0$ for all $k$. Moreover,
$\sum_{k=0}^{\infty} c(k) = \infty$ and $\sum_{k=0}^{\infty} c(k)^2 < \infty$.
\end{assumption}
A direct consequence of the last part of Assumption \ref{ass:ConvCoef} is that $\lim_{k \rightarrow \infty} c(k) = 0$. One choice for $\{c(k)\}_{k \geq 0}$ that satisfies the conditions of Assumption \ref{ass:ConvCoef} is to select it from the class of generalized harmonic series, e.g., $c(k) = \alpha/(k+1)$ for some $\alpha > 0$.
\begin{assumption} \label{ass:Weights}
There exists $\eta \in (0,1)$ such that for all $i,j \in \{1,\ldots,m\}$ and all $k \geq 0$, $a_j^i(k) \geq 0$, $a_i^i(k) \geq \eta$, and $a_j^i(k) > 0$ implies that $a_j^i(k) \geq \eta$.
Moreover, for all $k \geq 0$,
\begin{enumerate}
\item $\sum_{j=1}^{m} a_j^i(k) = 1$ for all $i=1,\ldots,m$,
\item $\sum_{i=1}^{m} a_j^i(k) = 1$ for all $j=1,\ldots,m$.
\end{enumerate}
\end{assumption}
The interpretation of having a uniform lower bound $\eta$, independent of $k$, for the coefficients $a_j^i(k)$ in Assumption \ref{ass:Weights} is that it ensures that each agent is mixing information received by other agents at a non-diminishing rate as iterations progress, \cite{Nedic_etal_2010}. Moreover, points 1 and 2 ensure that this mixing is a convex combination of the other agent estimates and the local estimate, where a non-zero weight is assigned to this latter since $a_i^i(k) \geq \eta$.

For each $k \geq 0$ the information exchange between the $m$ agents can be represented by a directed graph $(N,E_k)$, where the nodes $N = \{1,\ldots,m\}$ are the agents and the set $E_k$ of directed edges is given by
$E_k = \big \{ (j,i):~ a_j^i(k) > 0 \big \}$, i.e., at time $k$ agent $i$ receives information (namely, the estimate $x_j(k)$) from agent $j$, and this information is weighted by $a_j^i(k)$.
Let $E_{\infty} = \big \{ (j,i):~ (j,i) \in E_k \text{ for infinitely many } k \big \}$ denote the set of edges $(j,i)$ that represent agent pairs that communicate directly infinitely often. The following connectivity and communication assumption is eventually enforced.
\begin{assumption} \label{ass:Network}
The graph $(N,E_{\infty})$ is strongly connected, i.e., for any two nodes there exists a path of directed edges that connects them. Moreover, there exists $\bar{k} \geq 1$ such that for every $(j,i) \in E_{\infty}$, agent $i$ receives information from a neighboring agent $j$ at least once every consecutive $\bar{k} $ iterations.
\end{assumption}

Assumption \ref{ass:Network} guarantees that any pair of agents communicates at least indirectly infinitely often, and the intercommunication interval is bounded. For further details the reader is referred to \cite{consensus_paper_2016,Nedic_Ozdaglar_2009}.

\subsubsection{Convergence properties} \label{sec:secIIIC2}

Problem $\mathcal{P}$ can be equivalently written as
\begin{align}
\min_{x \in \bigcap_{i=1}^{m}X_i} \sum_{i=1}^m g_i(x), \label{eq:P_equiv}
\end{align}
where, for all $i=1,\ldots,m$, and for any $x \in \mathbb{R}^n$,
\begin{align}
g_i(x) = \min_{u_i \in U_i(x)} f_i(x,u_i). \label{eq:inner_opt}
\end{align}
Note that for all $x \in X_i$ the minimum in \eqref{eq:inner_opt} exists due to the Weierstrass' theorem (Proposition A.8, p. 625 in \cite{Bertsekas_Tsitsiklis_1997}), since $U_i(x)$ is compact by Assumption \ref{ass:Compact} and $f_i(\cdot,\cdot)$ is continuous due to Assumption \ref{ass:Convex}.
We then have the following auxiliary lemmas, which are crucial for the proof of Theorem \ref{thm:convergence}.
\begin{lemma} \label{lm:convexity}
Under Assumptions \ref{ass:Convex} and \ref{ass:Compact}, it holds that $g_i(\cdot): \mathbb{R}^n \to \mathbb{R}$ is convex on $X_i$, for all $i=1,\ldots,m$.
\end{lemma}

\begin{lemma} \label{lm:Lip_continuity}
Under Assumptions \ref{ass:Convex}, \ref{ass:Compact}, and \ref{ass:Lipschitz}, it holds that $g_i(\cdot): \mathbb{R}^n \to \mathbb{R}$ is Lipschitz continuous on $X_i$, for all $i=1,\ldots,m$.
\end{lemma}
The proofs of these technical lemmas are provided in the Appendix. \\

Consider now Algorithm 1, and, according to \eqref{eq:inner_opt}, re-write step 6 as
\begin{align}
x_i(k+1) = \arg\min_{x_i \in X_i} g_i(x_i) + \frac{1}{2 c(k)} ||\bar{x}_i(k) - x_i||^2. \label{eq:update_new}
\end{align}
Note that, since $g_i(\cdot)$ is convex on $X_i$ (Lemma \ref{lm:convexity}) and the quadratic penalty term in \eqref{eq:update_new} is strictly convex, $x_i(k+1)$ is univocally defined.

Based on this representation and building on Theorem 1 in \cite{consensus_paper_2016}, it can be deduced that Algorithm \ref{alg:Alg1} converges to a minimizer of $\mathcal{P}$. More precisely, it can be shown that there exists a minimizing global decision vector $x^*$ of $\mathcal{P}$ such that the values $\{ x_i(k) \}_{k\geq0}$ generated by Algorithm \ref{alg:Alg1} converge to $x^*$, for all $i=1,\ldots,m$ (i.e. agents reach consensus on the value of the global decision vector).
Moreover, though the local decision vector $\{ u_i(k) \}_{k\geq0}$, $i=1,\ldots,m$, generated by Algorithm \ref{alg:Alg1} may exhibit an oscillatory behavior, all their limit points will form together with $x^*$ a minimizer of $\mathcal{P}$.

This is formally stated in the following theorem, which is the main result of this section.

\begin{thm} \label{thm:convergence}
Let $\{x_i(k)\}_{k \geq 0}$, $\{u_i(k)\}_{k \geq 0}$, $i=1,\ldots,m$, be the sequences of esimates generated by Algorithm \ref{alg:Alg1}. Under Assumptions \ref{ass:Convex}-\ref{ass:Network}:
\begin{enumerate}
\item there exists a minimizing vector $x^*$ of $\mathcal{P}$, such that $\lim_{k \to \infty} || x_i(k) - x^*|| = 0$, for all $i=1,\ldots,m$;
\item any limit point $(u_1^*,\ldots,u_m^*)$ of $\{ (u_1(k),\ldots,u_m(k)) \}_{k \geq 0}$, is such that $(x^*,u_1^*,\ldots,u_m^*)$ is a minimizer of $\mathcal{P}$.
\end{enumerate}
\end{thm}

\begin{proof}
Consider Algorithm \ref{alg:Alg1} with step 6 rewritten as in \eqref{eq:update_new}. Thanks to Assumptions \ref{ass:Convex}-\ref{ass:Lipschitz} and thanks to Lemmas \ref{lm:convexity} and \ref{lm:Lip_continuity} it holds that
\begin{enumerate}[i)]
\item $X_i$ is convex and compact, for all $i=1,\ldots,m$,
\item $\bigcap_{i=1}^m X_i$ has non-empty interior,
\item $g_i(\cdot): \mathbb{R}^n \to \mathbb{R}$ is convex on $X_i$, for all $i=1,\ldots,m$,
\item $g_i(\cdot): \mathbb{R}^n \to \mathbb{R}$ is Lipschitz continuous on $X_i$, for all $i=1,\ldots,m$.
\end{enumerate}
Under these conditions\footnote{Actually, in \cite{consensus_paper_2016}, condition iv) is not imposed, but it is assumed that $g_i(\cdot)$ is convex on $\mathbb{R}^n$, and not only on $X_i$ as in condition iii). Convexity over the whole $\mathbb{R}^n$, together with the compactness condition in i), implies iv). Conditions iii) and iv) constitute a weaker set of assumptions, which, however, leave the conclusions of Theorem 1 in \cite{consensus_paper_2016} unaltered (see also discussion below Assumption 3 in \cite{consensus_paper_2016}).} and under Assumptions \ref{ass:ConvCoef}-\ref{ass:Network}, Theorem 1 in \cite{consensus_paper_2016} applies, yielding that there exists a minimizer $x^*$ of \eqref{eq:P_equiv}, such that $\lim_{k \to \infty} || x_i(k) - x^*|| = 0$, for all $i=1,\ldots,m$. By the equivalence between problem $\mathcal{P}$ and \eqref{eq:P_equiv}, $x^*$ is also the $x$-component of the minimizer of $\mathcal{P}$, and hence this concludes the proof for the first part of the theorem.

The second part follows along lines akin to the proof of point (b) of Theorem 1.17 in \cite{Rockafellar-Wets_2009}. Specifically, let $(u_1^*,\ldots,u_m^*)$ be any limit point of the sequence $\{ (u_1(k),\ldots,u_m(k)) \}_{k \geq 0}$, which exists thanks to the compactness Assumption \ref{ass:Compact}. Thanks to Assumption \ref{ass:Compact} it also holds that $(x^*,u_1^*,\ldots,u_m^*)$ is feasible for $\mathcal{P}$. Given the definition of $g_i$ and that of $u_i(k)$, recalling that $\lim_{k \to \infty} || x_i(k) - x^*|| = 0$ for all $i=1,\ldots,m$, and thanks to the continuity of $g_i(\cdot)$ as assured by Lemma \ref{lm:Lip_continuity}, for any given $\epsilon > 0$ it holds that
$$
\sum_{i=1}^m f_i(x_i(k),u_i(k)) = \sum_{i=1}^m g_i(x_i(k)) \leq \sum_{i=1}^m g_i(x^*) + \epsilon
$$
for $k$ large enough. This in turn implies that $\sum_{i=1}^m f_i(x^*,u_i^*) \leq \sum_{i=1}^m g_i(x^*) + \epsilon$. Being $\epsilon$ arbitrary, it follows that $\sum_{i=1}^m f_i(x^*,u_i^*) \leq \sum_{i=1}^m g_i(x^*)$, which, given the equivalence between problem $\mathcal{P}$ and \eqref{eq:P_equiv}, shows that $(x^*,u_1^*,\ldots,u_m^*)$ is a minimizer of $\mathcal{P}$.
\end{proof}

\subsection{Satisfaction of algorithm assumptions} \label{sec:secIIID}
The energy management problem of Section \ref{sec:secIIIA} is a convex minimization program. It can be also easily verified that its objective function is Lipschitz continuous with respect to all decision variables, thus satisfying Assumption \ref{ass:Convex}. Assumption \ref{ass:Compact} is also satisfied, as an effect of the physical and technological constraints imposed in
\eqref{eq:ch_energyReq}-\eqref{eq:init_storage}. Even if this were not the case, all numerical calculations are performed on compact domains, hence satisfaction of Assumption \ref{ass:Compact} is not an issue.

Assumptions \ref{ass:ConvCoef} and \ref{ass:Weights} imply that Algorithm \ref{alg:Alg1} is synchronous, and buildings need to agree prior to the execution of the algorithm on $\{c(\cdot)\}_{k \geq 0}$ and the weight coefficients $\{a^i_j(k)\}_{k \geq 0}$, $i,j = 1,\ldots,m$. For every iteration $k$, these weights should form a doubly stochastic matrix. A distributed methodology to construct doubly stochastic matrices can be found in \cite{Sinkhorn_Knopp_1967}; however, this is outside the scope of the current paper.
Assumption \ref{ass:Network} is standard in distributed optimization algorithms over networks, and is satisfied for a wide class of time-varying network structures. In particular, periodic absence of communication links, as in the case study of Section \ref{sec:secIV}, falls in the proposed framework.

Even though it is relatively straightforward to verify Assumptions \ref{ass:Convex}, \ref{ass:Compact}, and \ref{ass:ConvCoef}-\ref{ass:Network}, and they are satisfied for the case study of Section \ref{sec:secIV}, it is in general difficult to verify Assumption \ref{ass:Lipschitz}. This is due to the fact that existence of a uniform Lipschitz constant, such that the set-valued continuity condition \eqref{eq:Lipschitz_set} is satisfied, is hard to verify even numerically. However, applying Algorithm \ref{alg:Alg1} to the case study of Section \ref{sec:secIV} we verified numerically that the assertions of Theorem \ref{thm:convergence} are valid, even though we were not able to verify satisfaction of Assumption \ref{ass:Lipschitz}.

\section{Case study} \label{sec:secIV}
\subsection{Simulation set-up} \label{sec:secIVA}
Consider a network of $m=3$, identical, three-storey buildings, each one with a $20$m by $20$m base, a total height of $9$m, flat rooftop and half glazed lateral surfaces. Each building is divided into $n_z = 3$ thermal zones (one per floor) and is equipped with its own chiller, namely, building $1$ has a medium-size chiller, building $2$ a small one, and building $3$ a large one.
\begin{figure}[t!]
	\centering
	\includegraphics[width=\columnwidth]{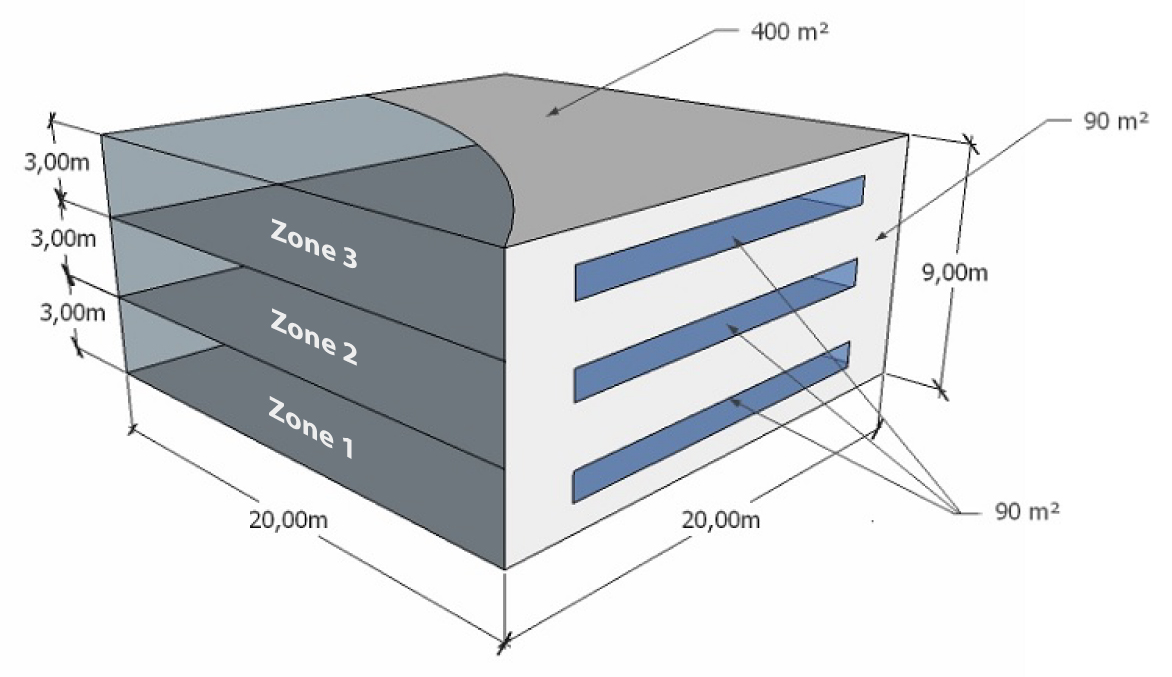}
	\caption{Structure of each building.}
	\label{fig:building}
\end{figure}

The structure of the buildings is schematically illustrated in Figure~\ref{fig:building} and the COP curves as a function of the cooling energy request $E_{\mathrm{chiller},c}$ are shown in Figure~\ref{fig:chillers}. The parameters of the biquadratic approximations (see Section \ref{sec:secIIB}) and the maximal cooling energy request are reported in Table~\ref{tab:chillers}, where the index $i$, $i=1,\ldots,3$, dictates the building they correspond to.
\begin{figure}[t!]
	\centering
	\includegraphics[width=\columnwidth]{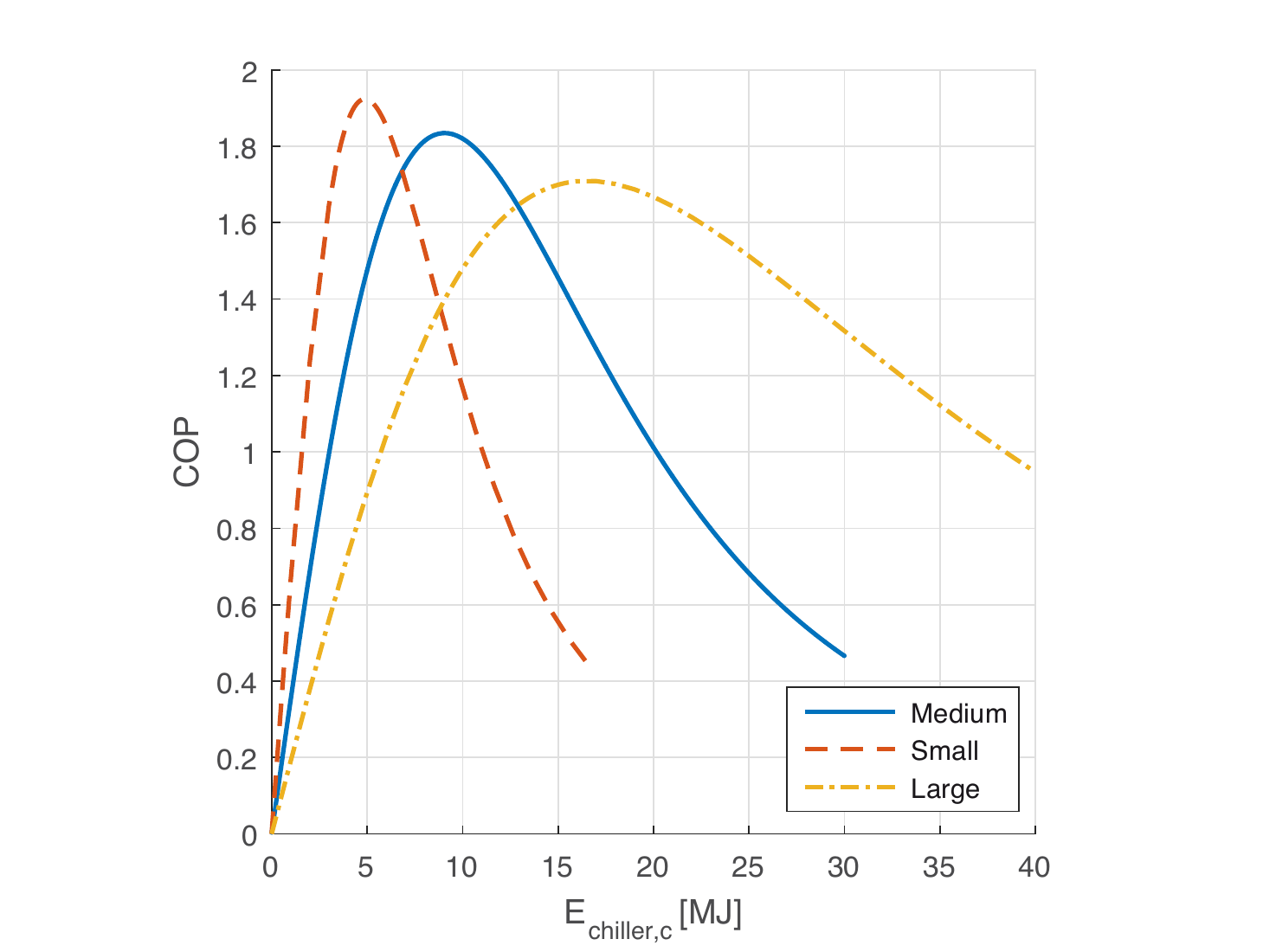}
	\caption{Chiller COP curves for each building.}
	\label{fig:chillers}
\end{figure}
\begin{table}[t!]
	\centering
	\begin{tabular}{cccccc}
		\hline \\[-0.85em]
		$i$ & Size         & $c_2^i$              & $c_1^i$              & $c_0^i$ & $E_{\max}^i$ {[}MJ{]} \\[0.15em]
		\hline \\[-0.85em]
		$1$ & Medium       & $3.79 \cdot 10^{-5}$ & $2.77 \cdot 10^{-2}$ & $2.46$  & $30$                  \\
		$2$ & Small        & $2.49 \cdot 10^{-4}$ & $4.98 \cdot 10^{-2}$ & $1.26$  & $18$                  \\
		$3$ & Large        & $3.56 \cdot 10^{-6}$ & $1.58 \cdot 10^{-2}$ & $5.11$  & $40$                  \\
		\hline \\
	\end{tabular}
	\caption{Chiller coefficients and maximal cooling energy.}
	\label{tab:chillers}
\end{table}

\begin{figure}[t!]
	\centering
	\includegraphics[width=\columnwidth]{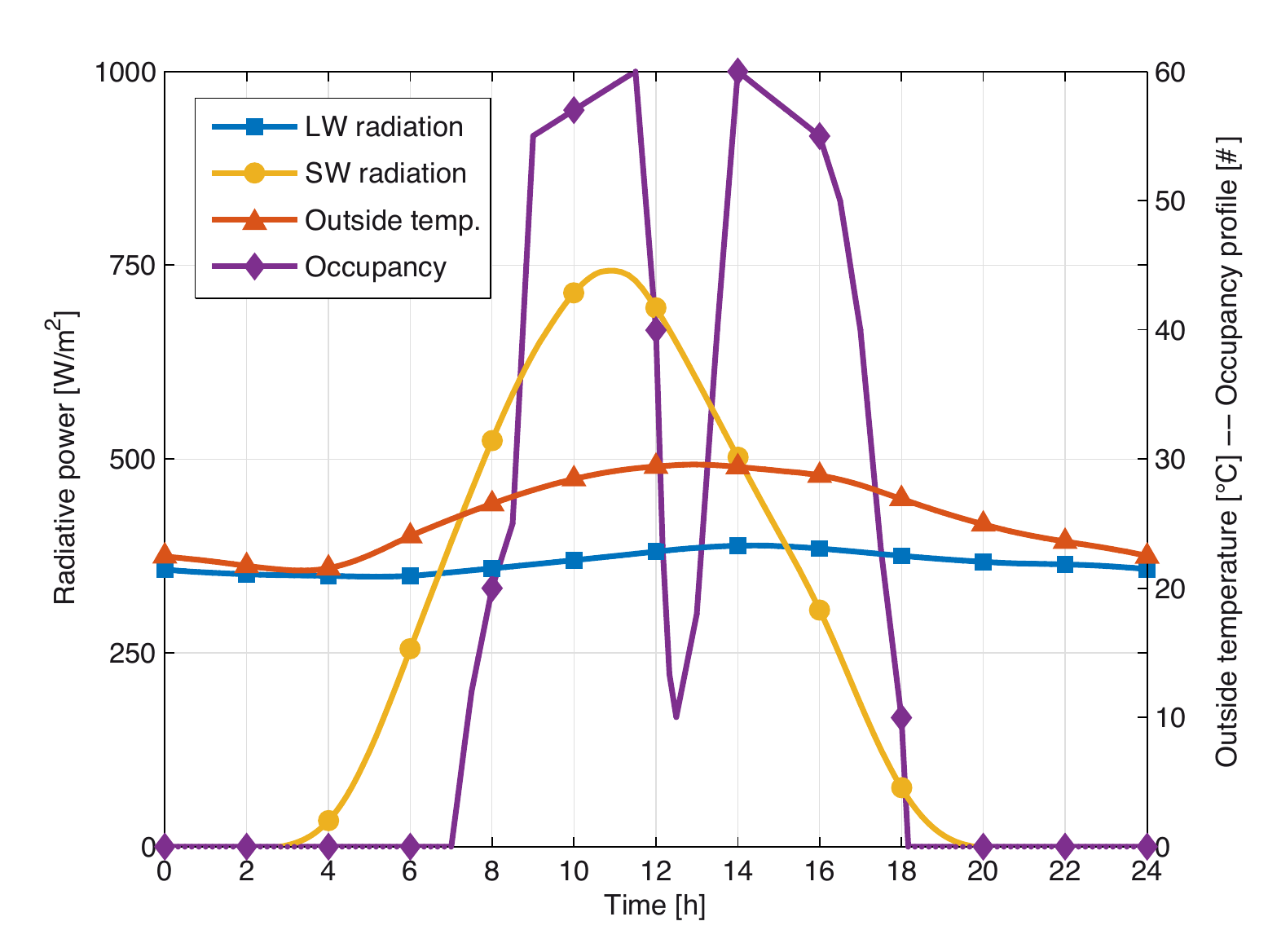}
	\caption{Disturbance profiles: Longwave (LW) and shortwave (SW) solar radiation, outdoor temperature and occupancy.}
	\label{fig:disturbances}
\end{figure}

\begin{figure}[t!]
	\centering
	\begin{tikzpicture}
		\def\RAD{0.3}
		\def\DIST{1.5}
		\def\ANGLE{360/3}
		
		\coordinate (O) at (0,0);
		
		\coordinate (N1) at ($(O)+(90+0*\ANGLE:\DIST)$);
		\coordinate (N2) at ($(O)+(90+1*\ANGLE:\DIST)$);
		\coordinate (N3) at ($(O)+(90+2*\ANGLE:\DIST)$);
		
		\draw[thick] (N1) -- node[xshift=-0.75em,yshift=0.25em] {$\frac{1}{3}$} (N2);
		\draw[thick] (N1) -- node[xshift=+0.75em,yshift=0.25em] {$\frac{1}{3}$} (N3);
		
		\filldraw[fill=white] (N1) circle (\RAD) node {$1$};
		\filldraw[fill=white] (N2) circle (\RAD) node {$2$};
		\filldraw[fill=white] (N3) circle (\RAD) node {$3$};
		
		\coordinate (O_2) at (4.25,0);
		
		\coordinate (N1_2) at ($(O_2)+(90+0*\ANGLE:\DIST)$);
		\coordinate (N2_2) at ($(O_2)+(90+1*\ANGLE:\DIST)$);
		\coordinate (N3_2) at ($(O_2)+(90+2*\ANGLE:\DIST)$);
		
		\draw[thick,blue]
			(N1_2) -- node[xshift=-0.75em,yshift=0.25em,black] {$\frac{1}{2}$} (N2_2);
		\draw[thick,red,style={decorate, decoration={snake, segment length=7, amplitude=2}}]
			(N2_2) -- node[xshift=0em,yshift=-1em,black] {$\frac{1}{2}$} (N3_2);
		\draw[thick,green!75!black,style={decorate, decoration={coil, aspect=1.25, segment length=5.6, amplitude=2}}]
			(N1_2) -- node[xshift=+0.75em,yshift=0.25em,black] {$\frac{1}{2}$} (N3_2);
		
		\filldraw[fill=white] (N1_2) circle (\RAD) node {$1$};
		\filldraw[fill=white] (N2_2) circle (\RAD) node {$2$};
		\filldraw[fill=white] (N3_2) circle (\RAD) node {$3$};
	\end{tikzpicture}
	
	\caption{Communication structure: Fixed (left panel) and time-varying (right panel).}
	\label{fig:communication_protocol}
\end{figure}
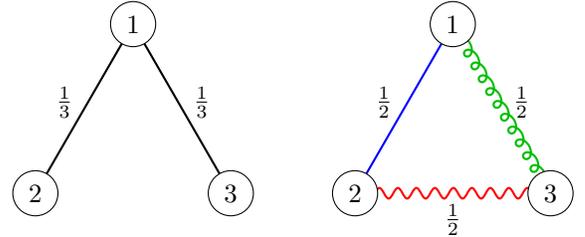

The external disturbances affecting the buildings are reported in Figure~\ref{fig:disturbances}. Longwave and shortwave solar radiation are depicted with square and circle markers respectively. The outside temperature and the occupancy are plotted with triangles and diamonds respectively. Note that the three buildings are supposed to be subject to the same disturbance profiles, and the occupancy shall be intended per building and equally partitioned among the zones. The period in which the occupancy is greater than zero is referred to as ``occupancy period'' and it is within the ``working hours'' range $7$AM to $6$PM. In all buildings, temperature constraints are set to $\widetilde{T}_{\min}^i = 20^\circ$C and $\widetilde{T}_{\max}^i = 24^\circ$C during working hours and to $\widetilde{T}_{\min}^i = 16^\circ$C and $\widetilde{T}_{\max}^i = 30^\circ$C otherwise.

For the control problem, we considered a time horizon of $24$ hours discretized in $n_t = 144$ time slots of $\dt = 10$min each. We tested the proposed algorithm with two different types of bi-directional communication topologies. The first one (Figure~\ref{fig:communication_protocol}, left panel) is a connected topology, in which buildings $1$ and $3$ exchange information only with building $2$ but not with each other, and the communication scheme is kept fixed across iterations. The second one (Figure~\ref{fig:communication_protocol}, right panel) is a time-varying periodic topology in which, at each iteration $k$, only two buildings communicate. The order in which the links are activated within the period is the following: $(1,2)$ (blue straight), $(2,3)$ (red wavy), and $(1,3)$ (green spring). In Figure~\ref{fig:communication_protocol} we also report the coefficients $a_j^i(k)$ for $j\neq i$ near the corresponding edge $(i,j)$. The $a_i^i(k)$ coefficients, $i=1,\dots,m$, are not reported but they can be easily retrieved so that Assumption~\ref{ass:Weights} is satisfied.

\subsection{Simulation results} \label{sec:secIVB}
We applied Algorithm~\ref{alg:Alg1} to the two communication structures and in both cases the proposed distributed approach was able to retrieve the optimal solution.

Figure~\ref{fig:final_temp} shows the optimal temperature profiles for the three zones of building $1$. It can be observed that, while the profiles of zones $1$ and $2$ are kept close to the maximum temperature bound of the working hours comfort range (outside the grey area), the temperature of zone $2$ is always lower than the other two. Zone $2$ is indeed subject to a pre-cooling phase before the occupancy period so as to cool down the building, acting as an additional passive thermal storage to drain the heat of the other zones through floor and ceiling. The temperature profiles of the other two buildings are very similar to that of building $1$, and hence are not reported here.
\begin{figure}[t!]
	\centering
	\includegraphics[width=\columnwidth]{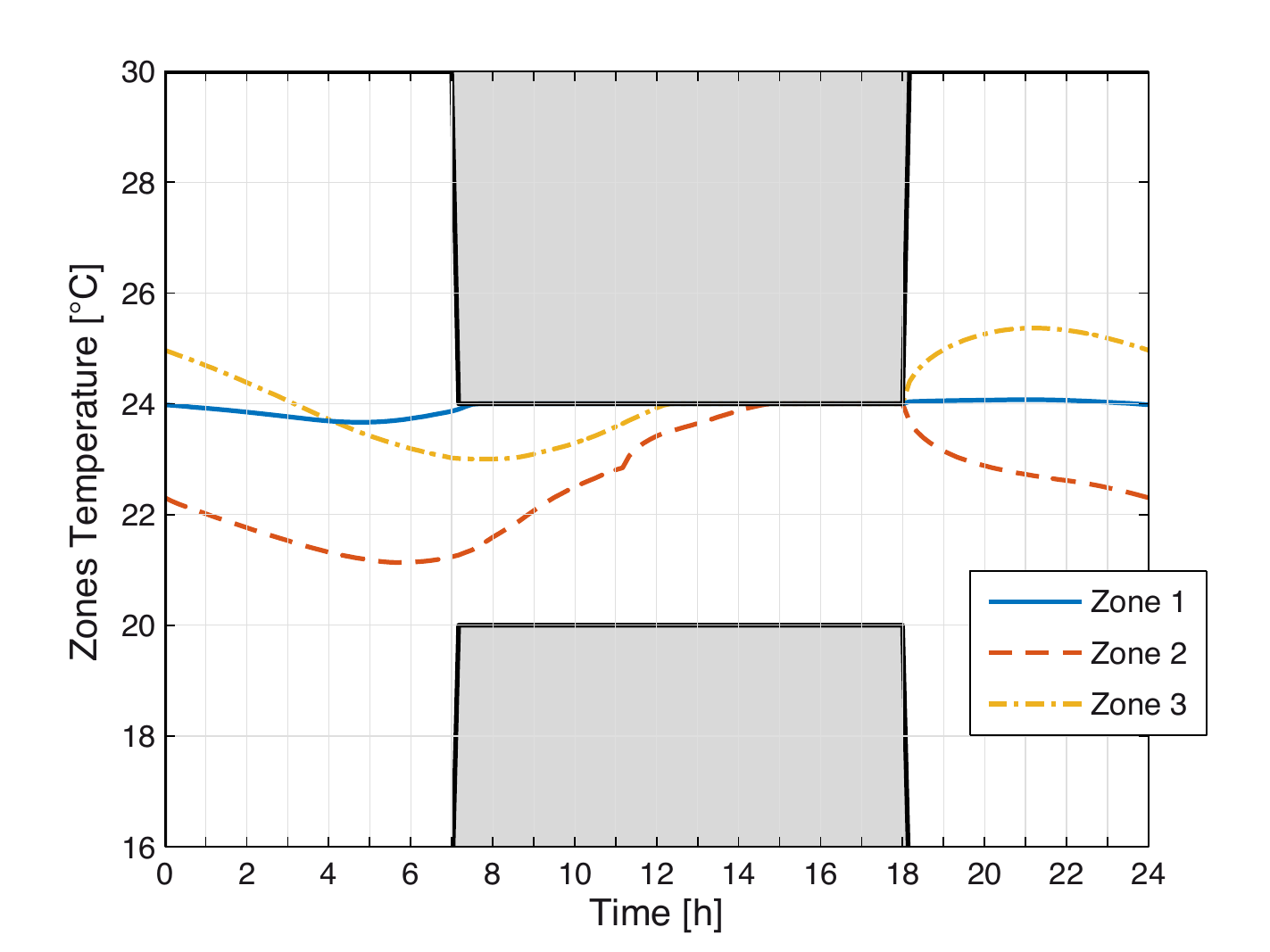}
	\caption{Optimal zone temperature profiles of building $1$. The temperature of zone $2$ (at the middle) is always lower than the other two, since it acts as a passive thermal storage to drain the heat of the other zones through floor and ceiling.}
	\label{fig:final_temp}
\end{figure}

In Figures~\ref{fig:init_storage} and \ref{fig:final_storage} we report the storage profiles of building $1$ at iteration $k=1$ and at consensus (when Algorithm \ref{alg:Alg1} converges), respectively. From Figure~\ref{fig:init_storage} it is clear that, at the beginning, building $1$ acts in a ``selfish'' manner and its optimal strategy is to constantly withdraw cooling energy from the storage ($e_s^1 > 0$, solid line), thus forcing buildings $2$ and $3$ to charge the storage ($e_s^2 < 0$ and $e_s^3 < 0$, dashed and dot-dashed lines, respectively). The stored energy is shown with the black dotted line. The consensus solution depicted in Figure~\ref{fig:final_storage} is instead cooperative. Building $3$, which has the biggest chiller, is constantly providing cooling energy ($e_s^3 < 0$) to the shared storage; building $2$, which has the smallest chiller, is constantly withdrawing energy ($e_s^2 > 0$) from it; and building $1$ provides/retrieves energy to/from the storage depending on the time slot. In this way, differences in the chiller sizes are compensated through the storage.
\begin{figure}[t!]
	\centering
	\includegraphics[width=\columnwidth]{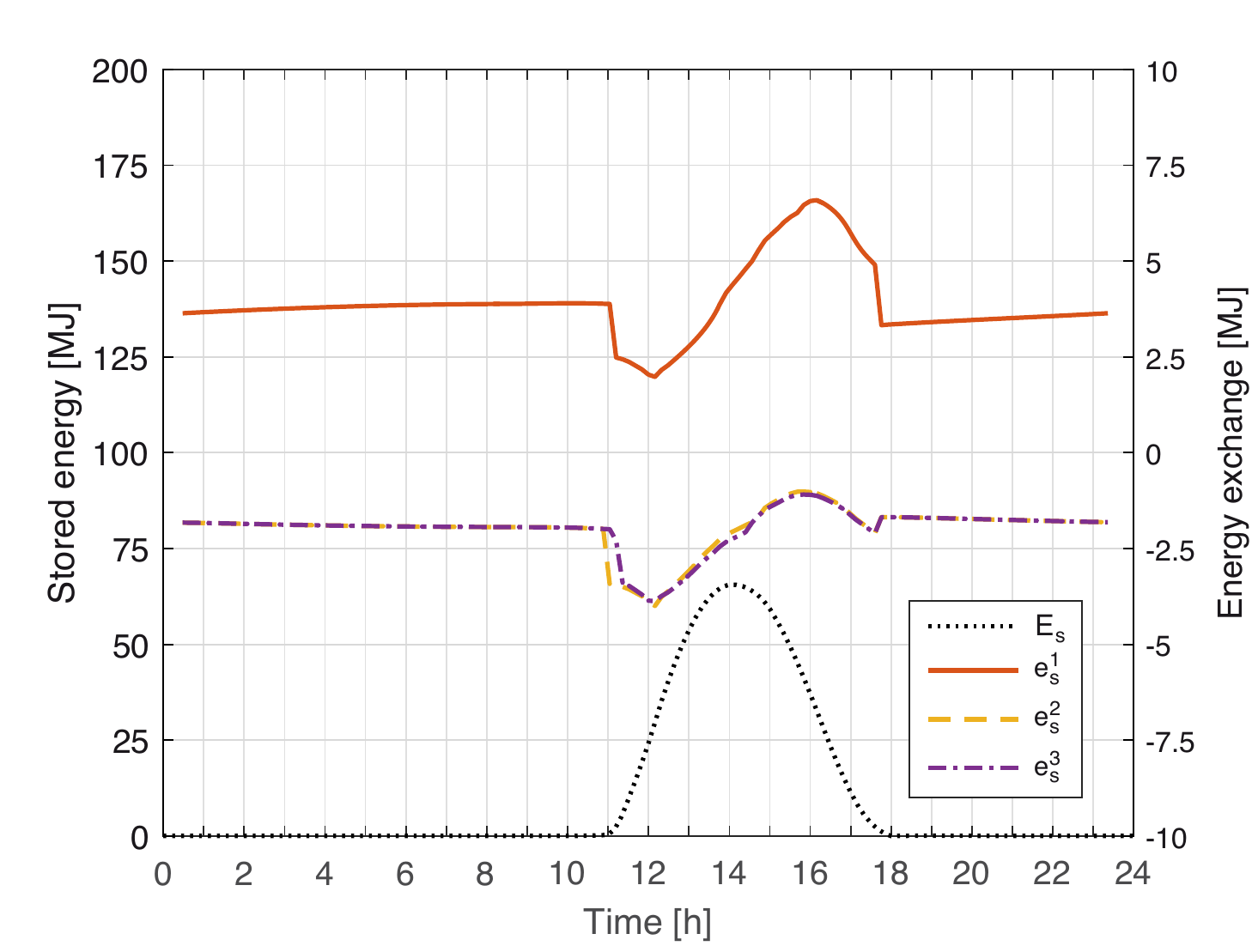}
	\caption{Storage profiles at iteration $k=1$. Building $1$ acts in a ``selfish''manner and its optimal strategy is to constantly withdraw cooling energy from the storage ($e_s^1 > 0$, solid line), thus forcing buildings $2$ and $3$ to charge the storage ($e_s^2 < 0$ and $e_s^3 < 0$, dashed and dot-dashed lines, respectively). The stored energy is shown with the black dotted line.}
	\label{fig:init_storage}
\end{figure}
\begin{figure}[t!]
	\centering
	\includegraphics[width=\columnwidth]{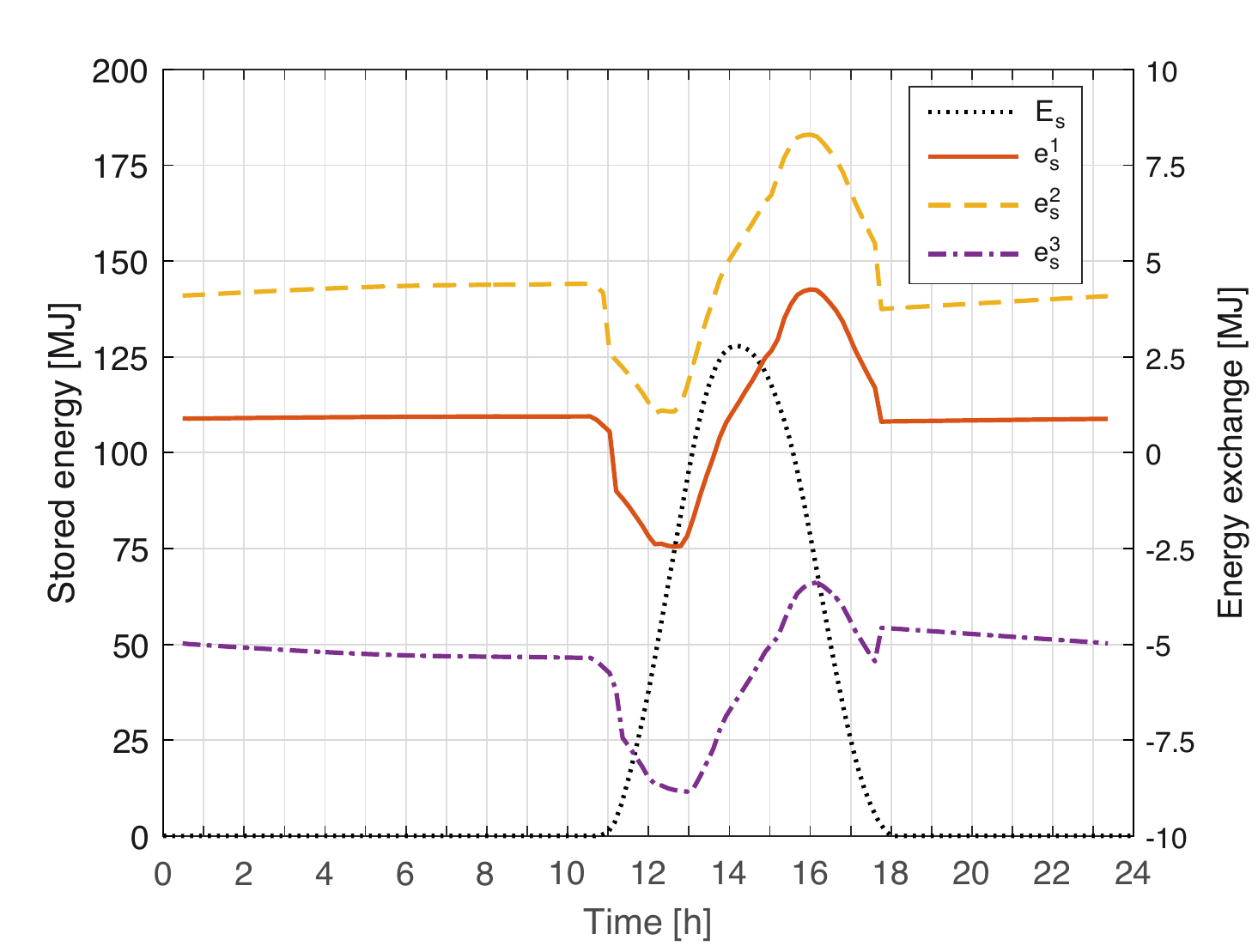}
	\caption{Storage profiles at consensus. Cooperative solution, with building $3$, which has the biggest chiller, is constantly providing cooling energy ($e_s^3 < 0$) to the shared storage; building $2$, which has the smallest chiller, is constantly withdrawing energy ($e_s^2 > 0$) from it; and building $1$ provides/retrieves energy to/from the storage depending on the time slot. The stored energy is shown with the black dotted line.}
	\label{fig:final_storage}
\end{figure}

Figure~\ref{fig:init_COP} and \ref{fig:final_COP} show the COP coefficient of the chillers of the three buildings (resulting from the optimization of building $1$) at $k=1$ and at consensus, respectively. In Figure~\ref{fig:init_COP} building $1$ is clearly optimizing the efficiency of its own chiller disregarding completely the efficiency of the other two, whereas the consensus solution reported in Figure~\ref{fig:final_COP} shows that the efficiency of the two other chillers is increased significantly at the expense of a slight deterioration in the one of building $1$, thus resulting in an overall benefit for the building district.
\begin{figure}[t!]
	\centering
	\includegraphics[width=\columnwidth]{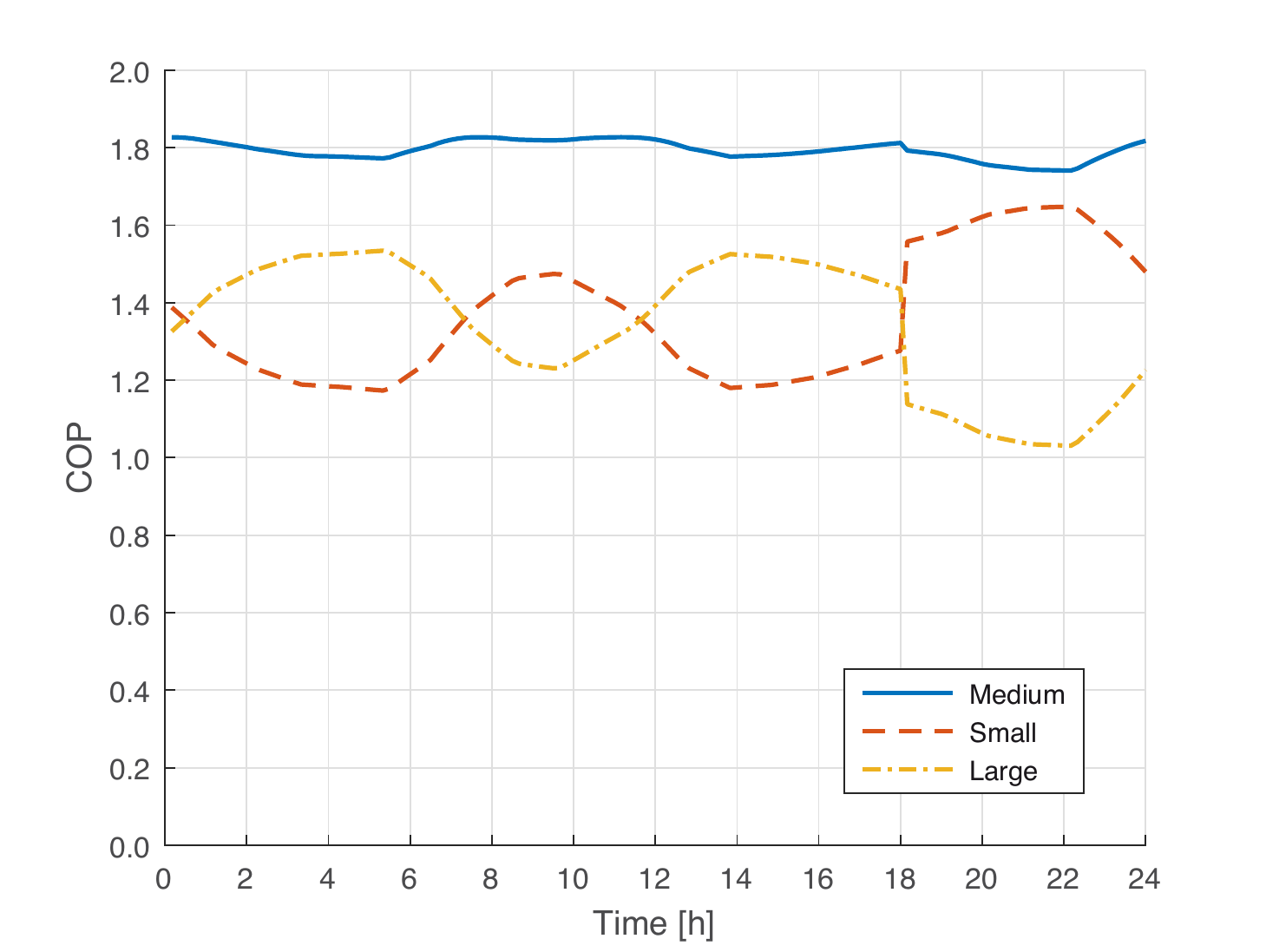}
	\caption{COP profiles at iteration $k=1$. Building $1$ is clearly optimizing the efficiency of its own chiller disregarding completely the efficiency of the other two.}
	\label{fig:init_COP}
\end{figure}
\begin{figure}[t!]
	\centering
	\includegraphics[width=\columnwidth]{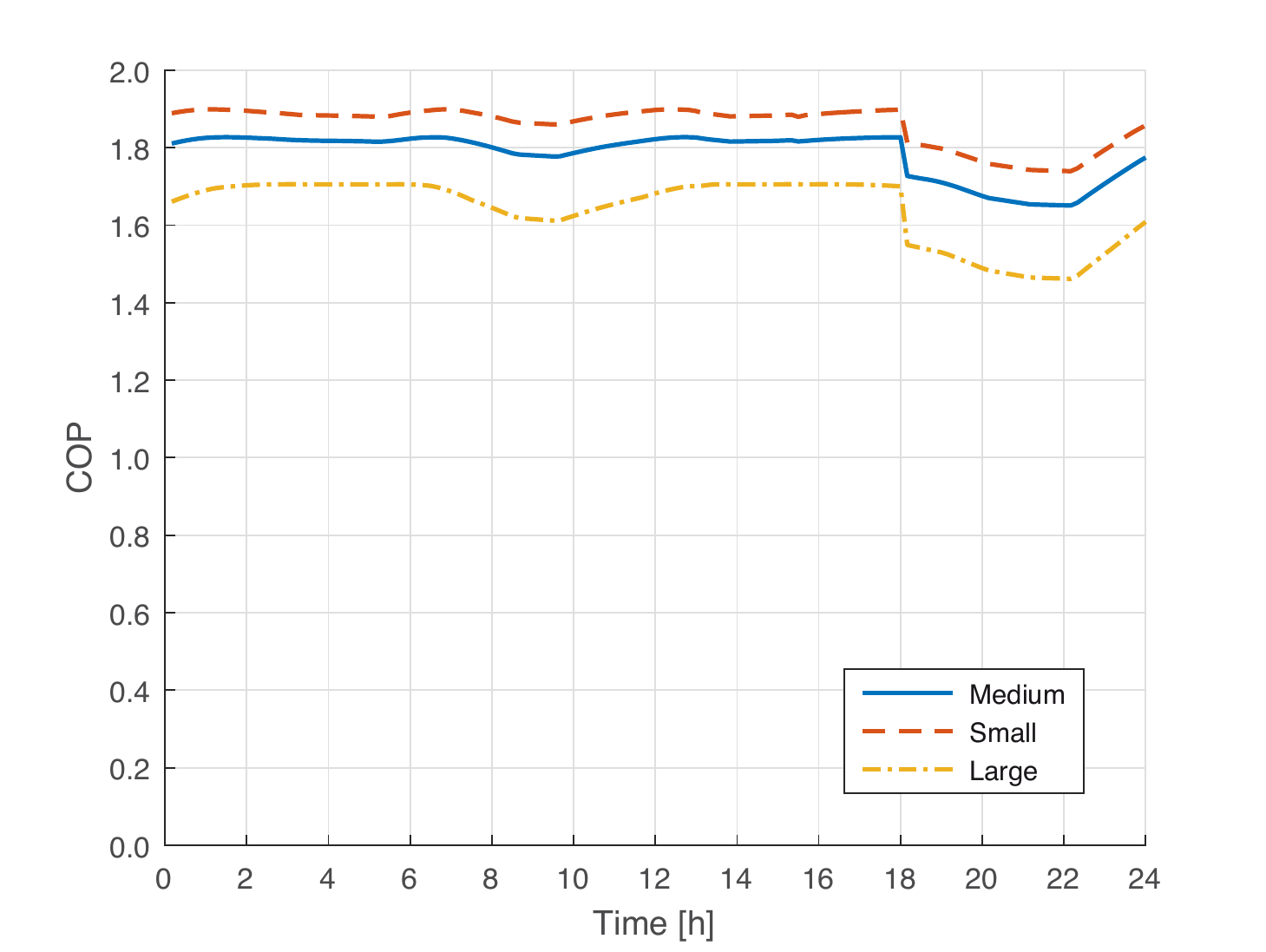}
	\caption{COP profiles at consensus. The efficiency of the chillers of buildings $2$ and $3$ is increased significantly at the expense of a slight deterioration in the one of building $1$, thus resulting in an overall benefit for the building district.}
	\label{fig:final_COP}
\end{figure}

The number of iterations needed to achieve consensus are $278$ for the fixed topology and $1032$ for the time-varying topology, where we considered the solution to be at consensus if either the absolute or the relative difference between the solutions of the agents across two consecutive iterations was less than a given threshold, which was taken to be $10^{-3}$.

\section{Concluding remarks} \label{sec:secV}
In this paper we proposed a distributed algorithm for energy management in buildings connected over time-varying networks, sharing common resources like storage. The proposed scheme does not require for buildings to reveal information that is considered as private, and overcomes the communication and computational challenges imposed by centralized or decentralized management paradigms.
In particular, a proximal minimization based approach was adopted, and a theoretical extension to an algorithm that recently appeared in the literature was provided. The efficacy of the proposed energy management algorithm was illustrated by means of a detailed simulation based study.

Current work concentrates on extending the proposed energy management scheme to take uncertainty due to renewable energy generation and/or occupancy into account. This would require extending our distributed algorithm to the stochastic case; preliminary theoretical results towards this direction, employing a scenario based approach, can be found in \cite{consensus_paper_2016}.

\section*{Appendix} \label{sec:secApp}

\begin{proofof}{Lemma \ref{lm:convexity}}
For each $i = 1,\ldots,m$, fix any $x, x' \in X_i$ and $\lambda \in [0,1]$.
By \eqref{eq:inner_opt}, let
\begin{align}
u_i^*(x) &\in \arg \min_{u_i \in U_i(x)} f_i(x,u_i), \label{eq:point_u1} \\
u_i^*(x') &\in \arg \min_{u_i \in U_i(x')} f_i(x',u_i). \label{eq:point_u2}
\end{align}
Note that the existence of such minimizers is guaranteed by Weierstrass' theorem (Proposition A.8, p. 625 in \cite{Bertsekas_Tsitsiklis_1997}), since $U_i(x), U_i(x')$ are compact and non-empty (Assumption \ref{ass:Compact}), and $f_i(\cdot,\cdot)$ is continuous (Assumption \ref{ass:Convex}).

Since $u_i^*(x) \in U_i(x)$ and $u_i^*(x') \in U_i(x')$, we have that $(x,u_i^*(x)) \in V_i$ and $(x',u_i^*(x')) \in V_i$, which, given the convexity of $V_i$ (Assumption \ref{ass:Compact}), implies that
\begin{align}
\big ( \lambda x + (1-\lambda) x', \lambda u_i^*(x) + (1-\lambda) u_i^*(x') \big ) \in V_i. \label{eq:conv_Vi}
\end{align}
This also implies that $\lambda u_i^*(x) + (1-\lambda) u_i^*(x') \in U_i( \lambda x + (1-\lambda) x' )$ (see \eqref{eq:con_u}).

We then have
\begin{align}
g_i( \lambda x  & + (1-\lambda) x' ) \nonumber \\
& = \min_{u_i \in U_i( \lambda x + (1-\lambda) x' )} f_i( \lambda x + (1-\lambda) x',u_i) \nonumber \\
& \leq f_i( \lambda x + (1-\lambda) x',\lambda u_i^*(x) + (1-\lambda) u_i^*(x') ) \nonumber \\
& \leq \lambda f_i(x,u_i^*(x)) + (1-\lambda) f_i(x',u_i^*(x')) \nonumber \\
& = \lambda g_i(x) + (1-\lambda) g_i(x'), \label{eq:convex_pf}
\end{align}
where the first inequality follows because $\lambda u_i^*(x) + (1-\lambda) u_i^*(x') \in U_i( \lambda x + (1-\lambda) x' )$ and the definition of $\min$, the second inequality because $f_i(\cdot,\cdot)$ is jointly convex with respect to its arguments (Assumption \ref{ass:Convex}), whereas the last equality because \eqref{eq:point_u1}, \eqref{eq:point_u2} and the definition of $g_i(\cdot)$ in \eqref{eq:inner_opt}. Since \eqref{eq:convex_pf} holds for any $x, x' \in X_i$, and for any $\lambda \in [0,1]$, the convexity of $g_i(\cdot)$ on $X_i$ remains proven.
\end{proofof}

\begin{proofof}{Lemma \ref{lm:Lip_continuity}}
The proof of is inspired by the proof of Corollary 3.5 of \cite{Wets_2003}.

For each $i=1,\ldots,m$, fix any $x, x' \in X_i$. Let also $u_i^*(x) \in U_i(x)$, $u_i^*(x') \in U_i(x')$, be as in \eqref{eq:point_u1} and \eqref{eq:point_u2}, respectively.
By Assumption \ref{ass:Lipschitz}, we have for all $u_i \in \mathbb{R}^{n_i}$ that
\begin{align}
\big | \min_{v_i \in U_i(x)} ||u_i - v_i|| -  \min_{v'_i \in U_i(x')} ||u_i - &v'_i|| \big | \nonumber \\ &\leq L_i ||x - x'||. \label{eq:hausdorff_pf1}
\end{align}
Take $u_i = u_i^*(x)$. We then have that
\begin{align}
\min_{v'_i \in U_i(x')} ||u_i^*(&x) - v'_i|| \nonumber \\
&\leq \min_{v_i \in U_i(x)} ||u_i^*(x) - v_i|| + L_i ||x - x'|| \nonumber \\
& \leq L_i ||x - x'||, \label{eq:hausdorff_pf2}
\end{align}
where the last inequality holds true because $u_i^*(x) \in U_i(x)$. Letting $\bar{v}'_i \in \arg \min_{v'_i \in U_i(x')} ||u_i^*(x) - v'_i||$, \eqref{eq:hausdorff_pf2} is equivalent to
\begin{align}
||u_i^*(x) - \bar{v}'_i|| \leq L_i ||x - x'||. \label{eq:hausdorff_pf5}
\end{align}

Similarly, taking $u_i = u_i^*(x')$ in \eqref{eq:hausdorff_pf1} gives that
\begin{align}
\min_{v_i \in U_i(x)} ||u_i^*(x') - v_i|| \leq L_i ||x - x'||, \label{eq:hausdorff_pf6}
\end{align}
which, letting $\bar{v}_i \in \arg \min_{v_i \in U_i(x)} ||u_i^*(x') - v_i||$ is equivalent to
\begin{align}
||u_i^*(x') - \bar{v}_i|| \leq L_i ||x - x'||. \label{eq:hausdorff_pf7}
\end{align}
Note that $\bar{v}_i, \bar{v}_i'$, exist due to the Weierstrass' theorem (Proposition A.8, p. 625 in \cite{Bertsekas_Tsitsiklis_1997}), since $U_i(x)$, $U_i(x')$ are compact and non-empty due to Assumption \ref{ass:Compact}, and since $||u_i^*(x') - v_i||$ and $||u_i^*(x) - v'_i||$ are continuous with respect to $v_i$ and $v'_i$, respectively.

By Assumption \ref{ass:Convex}, $f_i(\cdot,\cdot): \mathbb{R}^{n} \times \mathbb{R}^{n_i} \to \mathbb{R}$ is Lipschitz continuous. Denoting its Lipschitz constant by $C_i \in \mathbb{R}$, $C_i > 0$, we have that
\begin{align}
& f_i(x',\bar{v}_i') & \leq f_i(x, & u_i^*(x)) + C_i ||x - x'|| + C_i ||u_i^*(x) - \bar{v}_i'|| \nonumber \\
& & \leq f_i(x, & u_i^*(x)) + C_i (1+L_i) ||x - x'||, \label{eq:hausdorff_pf8}
\end{align}
where the last inequality follows in view of \eqref{eq:hausdorff_pf5}. Since $\bar{v}_i' \in U_i(x')$ and since $u_i^*(x')$ minimizes $f_i(x',\cdot)$ over $U_i(x')$, \eqref{eq:hausdorff_pf8} yields
\begin{align}
f_i(x',u_i^*(x')) \leq f_i(x,u_i^*(x))+ C_i (L_i + 1) ||x - x'||. \label{eq:hausdorff_pf9}
\end{align}

Similarly, by the Lipschitz continuity of $f_i(\cdot,\cdot)$ and by using  \eqref{eq:hausdorff_pf7}, we have that
\begin{align}
& f_i(x,\bar{v}_i) & \leq f_i(x', & u_i^*(x')) + C_i ||x - x'|| + C_i ||u_i^*(x') - \bar{v}_i|| \nonumber \\
& & \leq f_i(x', & u_i^*(x')) + C_i (1+L_i) ||x - x'||. \label{eq:hausdorff_pf10}
\end{align}
Since $\bar{v}_i \in U_i(x)$ and since $u_i^*(x)$ minimizes $f_i(x,\cdot)$ over $U_i(x)$, \eqref{eq:hausdorff_pf10} in turn gives that
\begin{align}
f_i(x,u_i^*(x)) \leq f_i(x',u_i^*(x'))+ C_i (L_i + 1) ||x - x'||. \label{eq:hausdorff_pf11}
\end{align}

Combining \eqref{eq:hausdorff_pf9} and \eqref{eq:hausdorff_pf11} we have that
\begin{align}
| f_i(x,u_i^*(x)) - f_i(x',u_i^*(x'))| \leq C_i (L_i + 1) ||x - x'||, \label{eq:hausdorff_pf12}
\end{align}
which is equivalent to
\begin{align}
| g_i(x) - g_i(x')| \leq C_i (L_i + 1) ||x - x'||, \label{eq:hausdorff_pf13}
\end{align}
being $g_i(x) = f_i(x,u_i^*(x))$ and $g_i(x') = f_i(x',u_i^*(x'))$.

Hence, $g_i(\cdot)$ is Lipschitz continuous on $X_i$ with Lipschitz constant $C_i (L_i + 1)$. This concludes the proof.
\end{proofof}


\end{document}